\PassOptionsToPackage{prologue,dvipsnames}{xcolor}

\documentclass[12pt]{amsart}
\usepackage[utf8]{inputenc}
\usepackage[T1]{fontenc}
\usepackage{fullpage}
\usepackage{tabu}
\usepackage{amssymb} 
\usepackage{amsmath} 
\usepackage{amscd}
\usepackage{tkz-graph}
\usepackage{amsbsy}
\usepackage{enumerate}
\usepackage[matrix,arrow]{xy}
\usepackage{hyperref}
\usepackage{mathrsfs}
\usepackage{color}
\usepackage{mathtools,caption}
\usepackage{todonotes}
\usepackage{tikz-cd}
\usepackage{tikz}
\usepackage{pbox}
\usepackage{tkz-graph}
\usetikzlibrary{decorations,fit,positioning,calc}
\usepackage{ocgx2}
\usepackage{texmaX,semmaX,semtkzX}
\usepackage{float}
\usepackage{makecell}
\usepackage{ulem}
\usepackage[overload]{textcase}

\def \ZZ{\mathbb{Z}}
\def \Z{\mathbb{Z}}
\def \QQ{\mathbb{Q}}
\def \Q{\mathbb{Q}}
\newcommand{\PP}{\mathbb{P}}

\newcommand{\PSL}{\operatorname{PSL}}
\renewcommand{\SL}{\operatorname{SL}}
\def \cO{\mathcal{O}}
\renewcommand{\Gal}{\operatorname{Gal}}
\renewcommand{\red}{\operatorname{red}}
\newcommand{\imag}{\operatorname{Im}}

\def\id#1{{\mathfrak{#1}}}  
\def \GL{\text{GL}}

\newcommand{\rhoJ}{\rho_{J, \lambda}}
\newcommand{\brhoJ}{\overline{\rho}_{J, \lambda}}
\newcommand{\rhoJf}{\rho_{J,\mathfrak{r}_5}}
\newcommand{\brhoJf}{\overline{\rho}_{J,\mathfrak{r}_5}}

\newcommand{\rhoJp}{\rho_{J, \mathfrak{p}}}
\newcommand{\brhoJp}{\overline{\rho}_{J, \mathfrak{p}}}
\newcommand{\brho}{\overline{\rho}}
\newcommand{\rhoJchip}{\rho_{J\otimes\chi_0,\id{p}}}
\newcommand{\brhoJchip}{\overline{\rho}_{J\otimes\chi_0,\id{p}}}
\newcommand{\rhogp}{\rho_{g,\id{P}}}
\newcommand{\brhogp}{\overline{\rho}_{g,\id{P}}}

\definecolor{darkgreen}{rgb}{0,0.5,0}

\numberwithin{equation}{section}

\theoremstyle{definition}

\theoremstyle{remark}

\usepackage{comment}

\begin{document}

\title{The Generalized Fermat equation \text{\texorpdfstring{$Ax^2 + By^r=Cz^p$}{AX2+BYr=Czp} and applications}}


\keywords{Fermat equations, modular method, Frey hyperelliptic curves, Darmon's program}
\subjclass[2020]{Primary 11D41, Secondary 11G10}

\author{Pedro-Jos\'e Cazorla Garc\'ia}
	
\address{Departamento de Matem\'atica Aplicada, ICAI, Universidad Pontificia Comillas, Madrid, 28015, Spain}  \email{pjcazorla@comillas.edu}

\author{Angelos Koutsianas}

\address{Department of Mathematics, Aristotle University of Thessaloniki\\
54124, Thessaloniki, Greece.} 
\email{akoutsianas@math.auth.gr}

\author{Lucas Villagra Torcomian}
	
\address{Department of Mathematics, Simon Fraser University, Burnaby, BC V5A 1S6, Canada.}  \email{lvillagr@sfu.ca}
	

\begin{abstract}
In this paper, we develop the modular method for the generalized Fermat equation appearing in the title, within the framework of Darmon’s program and using Frey hyperelliptic curves. As an application, we study a conjecture of Laradji, Mignotte, and Tzanakis concerning the equation $5x^2+q^{2n}=y^5$.
\end{abstract}

\maketitle


\section{Introduction}
\label{sec:introduction}

\subsection{The generalized Fermat equation} 

The groundbreaking proof of Fermat's Last Theorem \cite{Wiles} has sparked extensive research into the utilization of modular abelian varieties for studying classical Diophantine problems. The majority of this research is devoted to the resolution of \textit{generalized Fermat equations}, which are equations of the shape
\begin{equation}\label{eq:GFE}
    Ax^q+By^r=Cz^p,\qquad A, B, C\in\ZZ_{\neq 0}, \ q, r, p \in\ZZ_{> 0}
\end{equation}
to be solved for integers $a, b$ and $c$ and where the triple $(q,r,p)$ is called the \textit{signature} of the equation. If $a$, $b$ and $c$ satisfy that $\gcd(Aa, Bb, Cc) = 1$, we call the $(a,b,c)$ a \textit{primitive solution}. Under the assumption that $p$, $q$ and $r$ are fixed positive integers satisfying that
\begin{equation}
    \frac{1}{q}+\frac{1}{r}+\frac{1}{p}<1,
\end{equation}
Darmon and Granville \cite{Darmon-Granville} proved that \eqref{eq:GFE} has only a finite number of primitive solutions. In loc.\ cit.\ it was also conjectured that there are in fact at most finitely many primitive solutions, even when we allow the signature $(q,r,p)$ to vary\footnote{Here we are counting solutions as $1^q+2^3=3^2$ only once.}. For the case $A=B=C=1$, which has been broadly studied in the literature, this conjecture has been proved for some infinite families of signatures, such as $(p,p,p)$ \cite{Wiles}, $(p,p,2)$ and $(p,p,3)$ \cite{Darmon-Merel}, $(2,4,p)$ \cite{Ell,BennettEllenbergNg}, $(2,6,p)$ \cite{BennettChen}, etc. In all these cases, the methodology followed the strategy pioneered by Wiles, which is now referred as \textit{the modular method}.

\subsection{The modular method}
As previously mentioned, the modular method has been used to approach several Diophantine problems. For the convenience of the reader, we will give a high-level explanation of it, which is based on a contradiction argument. 

Without loss of generality, we shall assume that $p$ is a prime variable in \eqref{eq:GFE}. Then, assuming the existence of a putative \textit{non-trivial} solution $(a,b,c)$ to \eqref{eq:GFE} (i.e. $abc\neq0$), the main steps of the modular method are the following:

\begin{enumerate}
    \item  Attach to $(a,b,c)$ an abelian variety $A:=A(a,b,c)$  defined over some totally real field $K$. For all prime ideals $\id{p}$ of $K$ above $p$, $A$ should satisfy that the ramification set of the $\id{p}$-th member of its compatible system of Galois representations, $\bar{\rho}_{A,\id{p}}$, is independent of the solution.
    
    \item Prove that $A$ is \textit{modular}, i.e., that its $L$--series corresponds to the $L$--series of an automorphic form.
    
    \item Show that $\bar{\rho}_{A,\id{p}}$ is absolutely irreducible.

    \item Compute the conductor $\id{n}$  of $\bar{\rho}_{A,\id{p}}$.

    \item Compute the space of newforms of weight 2 and level $\id{n}$. For each newform, show that their associated residual Galois representation modulo $\id{p}$ is not isomorphic to $\bar{\rho}_{A,\id{p}}$ and therefore reach a contradiction.
\end{enumerate}

In the classical setting, the abelian variety of step (1) is an elliptic curve, traditionally known as a \textit{Frey elliptic curve}. At present, however, such a construction is available only for a restricted collection of signatures $(q,r,p)$ (see, for example \cite[Section~4]{Darmon-Granville} for a detailed discussion).

To extend the scope of the method and address new signatures, Darmon introduced the idea of replacing elliptic curves with higher-dimensional abelian varieties of $\GL_2$-type over a suitable totally real number field \cite{DarmonDuke}. This seminal contribution sparked the study of what is now known as \textit{Darmon's program}.

In the scope of this program, Jacobians of \textit{Frey hyperelliptic curves} have been increasingly used in step (1); see the survey~\cite{ChenKoutsianasSurvey} for a very complete description. So far, there are only a few families of signatures having explicit hypergeometric realization, namely: $(p,p,r)$, $(r,r,p)$, $(3,5,p)$  and $(2,r,p)$ (see \cite{DarmonDuke,BillereyChenDieulefaitFreitas25a,PacettiVillagra5p3,ChenKoutsianasSurvey}, respectively). 

\subsection{\texorpdfstring{Our contribution to signature $(2,r,p)$}{Our contribution to signature (2,r,p)}}

In this paper, we will focus on the study of the case of \eqref{eq:GFE} of signature $(2,r,p)$, namely
\begin{equation}\label{eqn:main}
    Ax^2 + By^r = Cz^p, \qquad \gcd(Ax,By,Cz)=1,
\end{equation}
with $r$ a fixed prime and $p$ a variable prime. In addition, we assume, without loss of generality, that $A$ is square-free and that $B$ is free of $r$-powers. Our motivations for considering this problem are two-fold. 

Firstly, $(2,r,p)$ is the remaining case, from the above families of signatures, where Darmon's program has not been carried out in full generality. {The abelian variety considered in step (1) is the Jacobian of a slight adaptation of the hyperelliptic curve introduced in \cite[Section 4.4]{ChenKoutsianas} (see also \cite[Section 6.3]{computationconductor}). In this article, we explain in detail why this Jacobian is of $\GL_2$-type over a totally real field, and hence gives rise to a two-dimensional Galois representation as required in step (1).

Up to now, steps (2) and (3) have not been addressed in the literature, and they will therefore be treated in the present paper.} Concerning step (4), we note that in \cite{computationconductor} the first and third authors compute the conductor exponent of the representation $\bar{\rho}_{A,\id{p}}$ at odd primes under the assumption that $A=1$.

In this paper we extend the results by computing the conductor exponent for odd primes for an arbitrary square-free value of $A$ (see Section \ref{Sec:conductor}), using the cluster picture methodology (described, for example, in \ \cite{BBB}).

In addition, under certain hypotheses, we show that the representation is irreducible if the exponent $p$ is large enough (see Corollary \ref{coro:large-image}), and therefore complete the steps of Darmon's program for \eqref{eqn:main} for arbitrary values of $A$, $B$ and $C$.

\subsection{Diophantine applications}

In a series of papers \cite{Rabinowitz78, Cao00, Muriefah00, Cohn03, Muriefah08, WangWang11, LaradjiMignotteTzanakis11} different authors have studied the Diophantine equation,
\begin{equation}\label{eq:application_general}
    rx^2 + q^{n} = y^r,\quad \gcd(x,y)=1,
\end{equation}
where $r,q$ are distinct primes and $n > 1$. 

The case $n=2\alpha$ has attracted the most attention in recent years. In \cite[Theorem 3.1]{Muriefah08} Abu Muriefah completely resolves \eqref{eq:application_general} when $r\not\equiv 7\pmod{8}$ and $r>3$ (with $n$ even). However, the proof contains a crucial mistake at its end \cite{Goss10}. In \cite{LaradjiMignotteTzanakis11} the authors are able to give a correct proof for the case $r\equiv 3\pmod{8}$ and partial results for the case $r=5$.

\begin{theorem}[{\cite[Theorem 1.1]{LaradjiMignotteTzanakis11}}]\label{thm:LaradjiMignotteTzanakis}
Let $q$ be an odd prime. If either $q\not\equiv 1\pmod{600}$ or $q\leq 3\cdot10^{9}$, then there is no integer solution $(x,y,\alpha)$ to the equation
\begin{equation}\label{eq:main_LaradjiMignotteTzanakis}
    5x^2 + q^{2\alpha}=y^5,\quad x,y,\alpha>0.
\end{equation}
Otherwise, there exists at most one integer solution $(x,y,\alpha)$ and if it actually exists, then it must satisfy the following conditions:
\begin{enumerate}
    \item $\alpha < 820$ and $\gcd(\alpha, 2\cdot 3\cdot 5\cdot 7\cdot 11\cdot 13)=1$,
    \item There exists an integer $v$ such that $x=10v(80v^4-40v^2+1)$, $y=20v^2 + 1$ and $q^\alpha=2000v^4-200v^2+1$.
\end{enumerate}
\end{theorem}

Due to the last theorem and experimental computations, the authors of \cite{LaradjiMignotteTzanakis11} also make the following conjecture.

\begin{conjecture}[{\cite[Conjecture p.\ 1576]{LaradjiMignotteTzanakis11}}]\label{con:application}
If the prime $q$ is not of the form $q=2000v^4-200v^2+1$, then \eqref{eq:main_LaradjiMignotteTzanakis} has no solutions.
\end{conjecture}

Motivated by Theorem \ref{thm:LaradjiMignotteTzanakis} and Conjecture \ref{con:application} we prove the following theorem as a Diophantine application of the theory developed in this article.

\begin{theorem}\label{thm:intro_application}
Let $p\ge 7$ be an odd prime with $p\neq 11$. There is no integer solution $(a,b,c)$ with $\gcd(a,b,c)=1$ of the equation \begin{equation}\label{eq:intro_diophantine_application}
    -5x^2 + y^5 = z^{2p},
\end{equation}
such that $2\mid a$ and $3\mid b$.
\end{theorem}

Theorems \ref{thm:intro_application} and \ref{thm:LaradjiMignotteTzanakis} answer Conjecture \ref{con:application} for the case $v\equiv \pm 1\pmod{3}$.

\begin{theorem}\label{thm:conjecture_prove}
Conjecture \ref{con:application} is correct when $v\equiv \pm 1\pmod{3}$.
\end{theorem}

It is important to mention that the case $v\equiv \pm 1\pmod{3}$ cannot be resolved by the methods in \cite[Section 3]{LaradjiMignotteTzanakis11} and the connection of Conjecture \ref{con:application} with the equation $5x^2-4=y^n$. Finally, the equation \eqref{eq:intro_diophantine_application} is a Fermat--type equation of signature $(2,5,p)$ with $A\neq 1$ and was therefore not amenable to the previously existing theory \cite{computationconductor}.

\subsection{Organization of the paper} The purpose of this article is to develop a complete theoretical framework for applying the modular method to generalized Fermat equations of signature $(2,r,p)$, drawing inspiration from Darmon’s program~\cite{DarmonDuke}. This goal is pursued throughout Sections~\ref{Sec:connectingfrey} to~\ref{sec:irred}, where we systematically revisit and adapt each step of the modular method outlined above. 

More specifically, relying on a correspondence between Frey representations of signatures $(p,p,r)$ and $(2,r,p)$ described in Section~\ref{Sec:connectingfrey}, we establish the existence of two-dimensional Galois representations associated with the Jacobian $J_{2,r}$ of the Frey hyperelliptic curve in Section~\ref{sec:Freycurve}. These representations are then used in Section~\ref{sec:modularity} to prove the modularity of $J_{2,r}$. In Section~\ref{Sec:conductor}, we study the conductor of the compatible system of $\lambda$-adic representations attached to $J_{2,r}$, while Section~\ref{sec:irred} provides a  criterion for proving asymptotic irreducibility of the corresponding residual representations under suitable hypotheses. Finally, in Section \ref{sec:application}, we consider the Diophantine equation \eqref{eq:intro_diophantine_application} discussed above and prove Theorem \ref{thm:intro_application} which implies the positive answer to the Conjecture \ref{con:application} in Theorem \ref{thm:conjecture_prove}.

\subsection{Notation} \label{sec:notation}
Throughout this paper we use the following notation.  Let $r\ge5$ be a fixed prime number. Then, we let
\begin{itemize}
    \item $\zeta_r$ denote a primitive $r$-th root of unity, 
    \item $K:=\Q(\zeta_r)^+=\Q(\zeta_r+\zeta_r^{-1})$ denote the maximal totally real subfield of $\Q(\zeta_r)$.
    \item $h(x)\in\Z[x]$ denote the minimal polynomial of $\omega := \zeta_r+\zeta_r^{-1}$.
    \item $\id{r}$ denotes the unique prime ideal of $K$ above $r$.
    \item Given a number field $L$, we denote by $\cO_L$ its ring of integers and by $G_L:=\text{Gal}(\overline{L}/L)$ its absolute Galois group.
\end{itemize}

\subsection{Electronic resources} The \texttt{Magma} \cite{magma}  programs used for the
computations in this article are posted in \url{https://github.com/akoutsianas/2rp}, which also includes a list of the programs,
their descriptions, output transcripts, timings, and the machines used.

\subsection{Acknowledgements}
The first author would like to thank Enrique González--Jiménez for insightful discussions. The second author was supported by the Special Account for Research Funds AUTH research grant \textit{``Solving Diophantine Equations-10349''}. The second author also wants to thank the first author for his great hospitality during his visit in Madrid. The third author was supported by a PIMS--Simons postdoctoral fellowship.

\section{\texorpdfstring{Connecting Frey representations of signatures $(p,p,p)$ and $(2,r,p)$}{Connecting Frey representations of signatures (p,p,p) and (2,r,p)}}
\label{Sec:connectingfrey}
\subsection{\texorpdfstring{Frey representations and abelian varieties of $\GL_2$-type}{Frey representations and abelian varieties of GL2-type}}

The following definition is a slight modification of Darmon’s original one (see \cite[Definition 1.1]{DarmonDuke}) adapted in \cite[Definition 2.1]{BillereyChenDieulefaitFreitas25a}. Here $M$  is a number field and $\F$ is a field of characteristic $p$ with algebraic closure $\overline{\F}$. Recall that the absolute Galois group of $\overline{M}(t)$, denoted by $G_{\overline{M}(t)}$, is a normal subgroup of the absolute Galois group of $M(t)$, $G_{M(t)}$.

\begin{definition}\label{def:FreyRep}
Let \( n_i \in \mathbb{N} \) and \( t_i \in \mathbb{P}^1(M) \) for $i = 1, \dots, m$. Then, a Frey representation in characteristic \( p \) over \( M(t) \)  of signature \( (n_i)_{i=1, \dots, m} \) with respect to the points \( (t_i)_{i=1, \dots, m} \) is a Galois representation
\begin{equation}
    \bar{\rho} : G_{M(t)} \to \GL_2(\F)
\end{equation}
satisfying:
\begin{enumerate}
    \item the restriction \( \bar{\rho}|_{G_{\overline{M}(t)}} \) has trivial determinant and is irreducible,
    \item the projectivization \( \mathbb{P}{\bar{\rho}}|_{G_{\overline{M}(t)}} : G_{\overline{M}(t)} \to \PSL_2(\F) \) of $\bar{\rho}|_{G_{\overline{M}(t)}}$ is unramified outside the \( t_i \),
    \item the projectivization \( \mathbb{P}{\bar{\rho}}|_{G_{\overline{M}(t)}}\) maps the inertia subgroups at \( t_i \) to subgroups of \( \PSL_2(\mathbb{F}) \) generated by elements of order \( n_i \), respectively, for \( i = 1, \dots, m \).
\end{enumerate}
\end{definition}

Let $\bar{\rho}_1$ and $\bar{\rho}_2$ be two Frey representations of the same signature $(n_i)_{i=1, \dots, m}$ with respect to the same points $(t_i)_{i=1, \dots, m}$. Then, we say that $\bar{\rho}_1$ and $\bar{\rho}_2$ are \textit{equivalent} if they are isomorphic over $\overline{\F}$ up to a character. In other words, there exists a character $\chi : G_M(t) \to \overline{\F}^\times$ such that the representations $\bar{\rho}_1 \otimes \chi$ and $\bar{\rho}_2 \otimes \chi$ are isomorphic.

\begin{definition}
    Let $F$ and $L$ be two fields of characteristic $0$ and let \( A \) be an abelian variety over \( L \). We say that \( A/L \) is of  $\GL_2$\textit{-type by} $F$ (or \( \GL_2(F) \)-type) if there exists an embedding \( F \hookrightarrow \text{End}_L(A) \otimes_{\mathbb{Z}} \mathbb{Q} \) where \( F \) is a number field with \( [F : \mathbb{Q}] = \dim A \).
\end{definition}

Let $A/L$ be an abelian variety of $\GL_2$-type by $L$. Then, it has associated a strictly compatible system of two-dimensional Galois representations (in the sense of \cite[Definition 4.10]{Bockle}) that we denote by $\{\rho_{A, \lambda}:G_L\to\GL_2(L_\lambda) \mid \lambda \text{ is a place of } L\}$ (see \cite[Section 11.10]{Shimura}).

\subsection{Frey hyperelliptic curves over function fields}

Our objective in this section is to relate Frey hyperelliptic curves defined over two function fields, $K(s)$ and $K(t)$. This will be useful in subsequent sections to construct a Galois representation that we will use in the modular method. For this purpose, we will closely follow the methodology in \cite[Section 2.4]{BillereyChenDieulefaitFreitas25a}.

Let $\Q(s)$ be the field of rational functions over $\Q$ in the variable $s$ and let $J(s)$ be the Jacobian of the hyperelliptic curve over $\Q(s)$ given by
\begin{equation}\label{eqn:Cr(s)}
    C_r(s) : y^2=(-1)^\frac{r-1}{2}xh(2-x^2)+s,
\end{equation}
where $r$ and $h(x)$ are as in Section \ref{sec:notation}. By \cite[Theorem 2.26]{BillereyChenDieulefaitFreitas25a}, for any element $s_0\in K$, $J(s_0)/K$ is of $\GL_2$-type by $K$, and so it has associated a strictly compatible system of two-dimensional Galois representations. Let $\id{p}$ be a prime of $K$ above a rational prime $p$ and consider the residual representation
\begin{equation}
    \bar{\rho}_{J(s),\id{p}} : G_K \to \GL_2(\F_\id{p})
\end{equation}
of the $\id{p}$-member of the family, where $\F_\id{p}$ is the residue field of $K_\id{p}$.
\begin{prop} \label{prop:FrepRep-ppr}
The representation $\bar{\rho}_{J(s),\id{p}}$ is a Frey representation of signature $(p,p,r)$ with respect to the points $(2,-2,\infty)$ in the sense of Definition \ref{def:FreyRep}.    
\end{prop}

\begin{proof}
If we set $s = 2-4t$, the curve $C_r(s)$ in \eqref{eqn:Cr(s)} becomes the curve $C_r^-(t)$ defined in \cite[p.\ 420]{DarmonDuke}. The associated representations are shown to be Frey representations of signature $(p,p,r)$ with respect to the points $(t_1, t_2, t_3) = (0,1,\infty)$ (see \cite[Theorem 1.10]{DarmonDuke}). The result then follows by finding the appropriate values of $s$ associated to the three values of $t$.
\end{proof}

Let $t\neq s$ be another variable. Let $K(s,t)$ be the function field where $s$ and $t$ are related by 
\begin{equation} \label{eq:relation-s-and-t}
    \frac{1}{s^2-4} = \frac{t-1}{4}.
\end{equation}

We can see both $K(s)$ and $K(t)$ as subfields of $K(s,t)$. Define $\alpha$ by the following equation 
\begin{equation}\label{eq:realtion-alpha}
    \alpha s = 2t
\end{equation}
and note that, from (\ref{eq:relation-s-and-t}), 
$\alpha$ is a square root of $t(t-1)$.

\begin{lemma}\label{lemma:twist2rp}
The quadratic twist by $\alpha$ of the base change of $C_r(s)$ to $K(s,t)$ has a model given by
\begin{equation}
    C_{2,r}(t) := y^2 = x^r + c_{r-2} \alpha^2 x^{r-2} + \dots + c_1 \alpha^{r-1} x + 2\alpha^{r-1}t,
\end{equation}
defined over $K(t)$. 
\end{lemma}

\begin{proof}
Because the polynomial defining the curve $C_r(s)$ has only odd powers of $x$ except for the constant term, we may apply the transformation $x \to \tfrac{x}{\alpha},$ $y \to \tfrac{y}{\alpha^{r/2}}$ to obtain, using (\ref{eq:realtion-alpha}), the model $C_{2,r}(t)$ of the quadratic twist by $\alpha$ of $C_r(s)$ over $K(t)$ given in the statement.
\end{proof}

Let $J_{2,r}(t)$ be the Jacobian of $C_{2,r}(t)$, which can be written as 
\begin{align}\label{eq:Freycurve-t}
C_{2,r}(t) : y^2 & = (-t(t-1))^\frac{r-1}{2}xh\Big(2-\frac{x^2}{t(t-1)}\Big) + 2(t-1)^\frac{r-1}{2}t^\frac{r+1}{2}.
\end{align}
Its discriminant equals
\begin{equation}\label{eq:disc-2rp}
    \Delta(C_{2,r}(t))      = (-1)^{\frac{r-1}{2}} 2^{3(r-1)} r^r t^\frac{r(r-1)}{2}(t-1)^\frac{(r-1)^2}{2}, 
\end{equation}
since this is a particular case of \cite[Lemma 3.1]{computationconductor} with $z = t(t-1)$ and $s = 2(t-1)^{(r-1)/2}t^{(r+1)/2}$.

\begin{proposition}
\label{prop:FreyRep}
Keeping the previous notation, the following holds:	\begin{enumerate}
    \item There is an embedding $K\hookrightarrow\End_{K(t)}(J_{2,r}(t))\otimes\Q$.
    \item For every specialization of $t$ to $t_0\in\PP^1(K)-\{0,1,\infty\}$ the above embedding is well-defined.
    \item For every prime $\id{p}$ of $K$ above a prime $p$, $\overline{\rho}_{J_{2,r}(t),\id{p}}:G_{K(t)}\to\GL_2(\F_\id{p})$ is a Frey representation of signature $(2,r,p)$ with respect to the points $(0,1,\infty)$.
\end{enumerate}
\end{proposition}

\begin{proof}
Replicating the arguments in \cite[Theorem 2.38]{BillereyChenDieulefaitFreitas25a}, conditions (1) and (2) follow from \cite[Theorem 2.26]{BillereyChenDieulefaitFreitas25a}, Proposition \ref{prop:FrepRep-ppr} and Lemma \ref{lemma:twist2rp}.

The ramification points $2, -2$ and $\infty$ of $C_r(s)$ correspond, respectively, to the ramification points $\infty$, $\infty$ and $1$ for $C_{2,r}(t)$, by (\ref{eq:relation-s-and-t}). This fact follows from \eqref{eq:relation-s-and-t} since the points $t = 1$ and $t = \infty$ give rise to $s = -2$ and $s = \infty$. In addition, from $(\ref{eq:disc-2rp})$ we see that $t = 0$ is a point of ramification for $C_{2,r}(t)$. Since $K(s,t)/K(t)$ is a degree $2$ extension and $t=0$ corresponds only to the point $s=0$, it follows that the projective order of inertia at $t=0$ is $2$. The orders at $t = 1$ and $t = \infty$ are $r$ and $p$, respectively. Hence, $J_{2,r}(t)$ gives rise to a Frey representation of signature $(2,r,p)$ with respect to the points $(0,1,\infty)$, finishing the proof of (3).
\end{proof}

\section{\texorpdfstring{A Frey hyperelliptic curve for signature $(2,r,p)$}{A Frey hyperelliptic curve for signature (2,r,p)}}\label{sec:Freycurve}

This section follows the steps in \cite[Section 3]{BillereyChenDieulefaitFreitas25a} in order to find a strictly compatible system of Galois representation suitable for the modular method. For this, we let $A,B$ be two fixed non-zero integers. For $a,b\in\Z$ such that $\gcd(Aa,Bb)=1$ and $ Aa^2+Bb^r\neq 0$, we set $C_{2,r}(a,b)$ to be the natural generalization of the hyperelliptic curve of genus $(r-1)/2$ constructed by Chen and Koutsianas in \cite[Section 4.4]{ChenKoutsianasSurvey} for $A=B=1$ (see also  \cite[Section 6.3]{computationconductor}). Namely, if $b\neq0$, we have
\begin{equation}\label{eqn:frey}
    C_{2,r}(a,b): y^2 = (ABb)^{\frac{r-1}{2}}xh\left(2+\frac{x^2}{ABb}\right) + 2A^{\frac{r+1}{2}}B^{\frac{r-1}{2}}a,
\end{equation}
 having discriminant equal to 
 \begin{align} \label{eq:disc2rp}
     \Delta(C_{2,r}(a,b))=(-1)^\frac{r-1}{2}2^{3(r-1)}r^rA^\frac{r(r-1)}{2}B^\frac{(r-1)^2}{2}(Aa^2+Bb^r)^\frac{r-1}{2}.
 \end{align}
 For $b=0$, coprimality implies that $a\neq0$ and we set
 \begin{align}
     C_{2,r}(a,b): y^2 = x^r + 2A^{\frac{r+1}{2}}B^{\frac{r-1}{2}}a.
 \end{align}
 Let $J_{2,r}(a,b)$ be the Jacobian of $C_{2,r}(a,b)$. Observe that if $b=0$, we have an endomorphism $(x,y)\mapsto(\zeta_rx,y)$ on $C_{2,r}(a,b)$ and hence $J_{2,r}(a,b)$ has complex multiplication (CM) by $\Q(\zeta_r)$. Moreover, if $a=0$ and $(a,b,c)$ is a solution to $(\ref{eqn:main})$ then $b=\pm1$ by coprimality and so $C_{2,r}(a,b)$ has CM by $\Q(i)$, since the curve has only odd terms and therefore $(x,y) \mapsto (-x,iy)$ is an endomorphism of $C_{2,r}(0, \pm1)$.

\begin{lemma}
    The curve $C_{2,r}(-a,b)$ is the quadratic twist by $-1$ of $C_{2,r}(a,b)$.
\end{lemma}
\begin{proof}
    It is clear for $b=0$ and it follows from \cite[Lemma 6.1]{computationconductor} for $b\neq 0$ by taking $\delta = -1$, $z = -ABb$ and $s = 2A^{(r+1)/2}B^{(r-1)/2}a$.
\end{proof}

In the next lemma we show that  the hyperelliptic curve $C_{2,r}(a,b)$ is related to a specialization of the model $C_{2,r}(t)$ given in (\ref{eq:Freycurve-t}). Let $z_0$ be a fixed square root of $-ABb$. We specialize the variables $s$ and $t$ in (\ref{eq:relation-s-and-t}) to be $s=s_0$ and $t=t_0$, where
\begin{align}\label{eqn:t0s0}
t_0=\frac{Aa^2}{Aa^2+Bb^r} \quad \text{ and  } \quad s_0= \frac{2A^{\frac{r+1}{2}}B^{\frac{r-1}{2}}a}{z_0^r}.
\end{align}
Note that, if $ab \neq 0$ and $Aa^2+Bb^r \neq 0$, we have that $t_0\in\PP^1(\Q)\setminus\{0,1,\infty\}$ and that 
\begin{equation}
    t_0(t_0-1)=-\frac{Aa^2Bb^r}{(Aa^2+Bb^r)^{2}}.
\end{equation}

\begin{lemma}\label{lemma:quadtwist}
Assume that $ab\neq0$. Then   $C_{2,r}(t_0)$ is the quadratic twist of  $C_{2,r}(a,b)$ by $\pm\frac{b^\frac{r-1}{2}a}{Aa^2+Bb^r}$.
\end{lemma}

\begin{proof}
Note that $C_{2,r}(a,b)$ can be obtained by $C_r(s_0)$ by replacing $x$ by $x/z_0$ and $y$ by $y/z_0^{r/2}$, showing that $C_{2,r}(a,b)$ is the quadratic twist of $C_r(s_0)$ by $z_0$. Also, we may check that $t_0$ and $s_0$ as defined in \eqref{eqn:t0s0} satisfy \eqref{eq:relation-s-and-t} and, consequently, Lemma \ref{lemma:twist2rp} gives that $C_{2,r}(t_0)$ is the quadratic twist of $C_r(s_0)$ by $\alpha$, where $\alpha$ is the square-root of $t_0(t_0-1)$. Therefore, by composing both twists, we conclude that the curve $C_{2,r}(t_0)$  is the quadratic twist of  $C_{2,r}(a,b)$ by $\alpha/z_0=\pm b^\frac{r-1}{2}a/(Aa^2+Bb^r)$. 
\end{proof}

Denote by $J_{2,r}$  the base change of $J_{2,r}(a,b)$ to $K$. 
\begin{theorem}\label{thm:2rp-GL2-type}
Assume that $b\neq0$. Then $J_{2,r}$ is of $\GL_2(K)$-type. In particular, it gives rise  to a strictly compatible system of $K$-integral $\lambda$-adic representations
\begin{align*}
    \rho_{J_{2,r},\lambda}:G_K\to \GL_2(K_\lambda).
\end{align*}
\end{theorem}

\begin{proof}
This follows from Proposition \ref{prop:FreyRep} and Lemma \ref{lemma:quadtwist}.
\end{proof}

\section{Modularity}\label{sec:modularity}

Keep the notation from previous sections. In particular, let $J_{2,r}$ be the base change of $J_{2,r}(a,b)$ to $K$ and let $J_{2,r}(t)$ be the Jacobian of $C_{2,r}(t)/K$. Recall that we will specialize at $t_0$ as given in \eqref{eqn:t0s0}. Finally, as in Section \ref{sec:notation},  let $\id{r}$ be the unique prime of $K$ above $r$.

From Theorem \ref{thm:2rp-GL2-type}, there is an strictly compatible system of 2-dimensional Galois representations $\{\rho_{J_{2,r},\lambda}:G_K\to \GL_2(K_\lambda)\}$ attached to $J_{2,r}$.

\begin{theorem}\label{thm:residualtwistC2}
The residual representation $\bar{\rho}_{J_{2,r},\id{r}}$ is isomorphic to a twist of $\bar{\rho}_{C_2(1/t_0),r}$, where $C_2(t)$ is Darmon's curve, given by  $C_2(t):y^2=x^3+2x^2+tx$.

Moreover, the representation $\bar{\rho}_{J_{2,r},\id{r}}$ extends to $G_\Q$.
\end{theorem}
\begin{proof}
From Proposition \ref{prop:FreyRep}(3), and because we are taking $\id{p} = \id{r}$, we know that $\bar{\rho}_{J_{2,r}(t),\id{r}}$ is a Frey representation over $K$ of signature $(2,r,r)$  with respect to the points $(0,1,\infty)$. Writing $G_{\Q(t_0)}$ and $G_{K(t_0)}$ for $G_\Q$ and $G_K$ respectively to make clear that the actions involved are obtained by specialization, it follows from \cite[Theorem 1.14 and \S1.3] {DarmonDuke} (see also \cite[Example 4.8]{ChenKoutsianasSurvey}) and specialization at $t=t_0$ that, as representations of $G_{K(t_0)}$, we have
\begin{equation}\label{eq:congC2}
    \overline{\rho}_{J_{2,r}(t_0),\mathfrak{r}} \simeq \overline{\rho}_{C_2(1/t_0),r} \otimes \chi
\end{equation}
where $\chi : G_{K(t_0)} \to \overline{\F}_r^\times$ is a character of order at most $2$. We note that both $J_{2,r}(t_0)$ and $C_2(1/t_0)$ are non-singular as $t_0 \neq 0,1,\infty$.  
The same conclusion holds for $\overline{\rho}_{J_{2,r},\mathfrak{r}}$ by Lemma \ref{lemma:quadtwist} and the expression of the discriminant in \eqref{eq:disc2rp}, completing the proof of the first assertion.  

The second statement follows by the same argument that in the proof of \cite[Theorem 4.1]{BillereyChenDieulefaitFreitas25a}.
\end{proof}

\begin{theorem}\label{thm:residualIrred}
Assume that $r \geq 11$. In addition, suppose that:
\begin{enumerate}
    \item $Aa^2+Bb^r\neq \pm 1$,
    \item $(Aa^2, Bb^r)\notin\{(2^i\pm1,\mp 1) : 0\le i\le 6\}$.
\end{enumerate}
Then the representation $\overline{\rho}_{J_{2,r},\mathfrak{r}}$ is absolutely irreducible when restricted to $G_{\Q(\zeta_r)}$.
\end{theorem}

\begin{proof}
From the proof of Theorem \ref{thm:residualtwistC2} we know that $\overline{\rho}_{J_{2,r},\mathfrak{r}}$ satisfies
\begin{equation}\label{eq:45}
    \overline{\rho}_{J_{2,r},\mathfrak{r}} \simeq \overline{\rho}_{C_2,r} \otimes \chi
\end{equation}
as $G_K$-representations, where $\chi : G_K \to \overline{\F}_r^\times$ is a character of order at most $2$ and $C_2=C_2(1/t_0)$. Note that $C_2/\Q$ is an elliptic curve (it is non-singular as $t_0(t_0-1)\neq 0$) with a $2$-torsion point over $\Q$ and $j$-invariant equal to $j(C_2)=2^6(4t_0-3)^3/(t_0-1)$.  

By \eqref{eq:45}, the representation $\overline{\rho}_{J_{2,r},\mathfrak{r}}|_{G_{\Q(\zeta_r)}}$ is absolutely irreducible if and only if $\overline{\rho}_{C_2,r}|_{G_{\Q(\zeta_r)}}$ is absolutely irreducible. Note that $\bar{\rho}_{C_2,r}(G_{\Q(\zeta_r)})=\bar{\rho}_{C_2,r}(G_\Q) \cap \SL_2(\F_r)$.

Firstly, suppose that $C_2$ does not have  CM. Then, by \cite[Proposition 3.1]{Najman}, we have that, if $r \ge 11$, $\bar{\rho}_{C_2, r}(G_\Q) = \GL_2(\F_r)$ and so $\bar{\rho}_{C_2,r}(G_{\Q(\zeta_r)})=\SL_2(\F_r)$. Thus, the representation is surjective and absolutely irreducible.

Secondly, suppose that $C_2$ has CM. In particular, its $j$-invariant must be integral, that is,
\begin{equation}
    j(C_2)=-2^6\cdot\frac{(Aa^2-3Bb^r)^3}{(Aa^2+Bb^r)^2Bb^r}\in\ZZ.
\end{equation}
Since $\gcd(Aa,Bb)=1$, $Bb^r\mid 2^6$, that is, $Bb^r=2^i$, for some $0\le i \le 6$. If $i>0$ then $Aa^2+Bb^r$ is odd and, since $\gcd(Aa^2+Bb^r, Aa^2-3Bb^r) \mid 4$, we get that $Aa^2+Bb^r=\pm 1$, which gives a contradiction to assumption (1). Otherwise, $Bb=\pm 1$ and so $Aa^2+Bb^r \mid 2^6(Aa^2-3Bb^r)$, meaning that $Aa^2\pm 1 \mid 2^6$, which gives a contradiction to assumption (2).
\end{proof}

\begin{remark} Suppose that $A=B=1$ and let $(a,b,c)$ be a solution to (\ref{eqn:main}). Then, since $r \ge 5$, by Catalan's conjecture (proven in \cite{MR2076124}), assumption (1) holds,  except for the trivial cases, i.e.\ when $ab=0$. 
\end{remark}

\begin{lemma}\label{lemma:irred-r=7}
Assume that $r=7$ and $(Aa^2,Bb^r)\notin\{(\pm 63, \pm 1), (\pm 63, \mp 64)\}$. Then the representation $\overline{\rho}_{J_{2,r},\mathfrak{r}}$ is absolutely irreducible when restricted to $G_{\Q(\zeta_7)}$.
\end{lemma}

\begin{proof} 
We use again the isomorphism (\ref{eq:congC2}). Suppose $\bar{\rho}_{C_2,7}$ is reducible. Since $C_2$ has a 2-torsion point, it gives rise to a non-cuspidal rational point on the curve $X_0(14)$, which has only two non-cuspidal points, corresponding to the $j$-invariants $-3375$ and $16581375$. These give rise to the values $t_0=63/64$ and $t_0=-63$. Since $t_0=Aa^2/(Aa^2+Bb^r)$, this gives $(Aa^2, Bb^r) \in \{(\pm 63, \pm1), (\pm 63, \mp 64)\}$, thereby contradicting the hypothesis in the statement of the lemma. Consequently, $\overline{\rho}_{C_2,7}$ is irreducible over $\Q$. Now, let $k$ denote its Serre's weight. Since $C_2$ is semistable at $7$, either $k=2$ or $k=8$. If $\bar{\rho}_{C_2,7}$ restricted to $G_{\Q(\zeta_7)}$ were not absolutely irreducible, \cite[Lemmas 1.13 and 1.14]{DieulefaitPacetti} gives $k=4$ or $k=5$, giving a contradiction.
\end{proof}

\begin{theorem} \label{thm:modularity} 
Let $r\neq 5$ and let $(a,b)$ satisfying the following:
\begin{enumerate}
    \item If $r=7$, assume that $Aa^2+Bb^r \not \in \{(\pm 63, \pm1), (\pm 63, \mp 64)\}$,
    \item If $r\ge 11$, assume that $Aa^2+Bb^r \neq \pm 1$ and $(Aa^2, Bb^r)\notin\{(2^i\pm1,\mp 1) : 0\le i\le 6\}$.
\end{enumerate}
 Then the abelian variety $J_{2,r}=J_{2,r}(a,b)/K$ is modular.
\end{theorem}

\begin{proof}
For $r=3$ we have $K=\Q$  and $J_{2,r}$ is the rational elliptic curve given in \cite[p.\ 530]{Darmon-Granville}, and so it is modular by the modularity theorem \cite{Wiles, Taylor-Wiles}. We now assume $r \ge 7$. By Theorem \ref{thm:residualtwistC2} the representation $\overline{\rho}_{J_{2,r},\id{r}}:G_K \to \GL_2(\F_r)$ extends to an odd representation $\bar{\rho}$ of $G_\Q$ which is absolutely irreducible as a consequence of Theorem \ref{thm:residualIrred} for $r\ge11$ and Lemma \ref{lemma:irred-r=7} for $r=7$. From Serre's Conjecture, the representation $\bar{\rho}$ is modular, hence $\bar{\rho}_{J_{2,r},\id{r}}$  is also modular by cyclic base change. Modularity of $\rho_{J_{2,r},\id{r}}$ now follows from \cite[Theorem 1.1]{KhareThorne}, by checking its three conditions:
\begin{enumerate}
    \item $\rho_{J_{2,r},\id{r}}$ is  unramified outside $\Delta(C_{2,r})$.
    \item The abelian variety $J_{2,r}$ is  potentially semistable, so $\rho_{J_{2,r},\id{r}}$ is de Rham and their  Hodge--Tate  weights are $\{0,1\}$.
    \item The representation $\rho_{J_{2,r},\id{r}}|_{G_{\Q(\zeta_r)}}$ is  absolutely irreducible from Theorem \ref{thm:residualIrred} for $r\ge11$ and from Lemma \ref{lemma:irred-r=7} for $r=7$.
\end{enumerate}
    
Since $\rho_{J_{2,r},\id{r}}$ is modular, then the statement follows by \cite[Theorem (1.3.1)]{RibetGalois}.
\end{proof}

For $r=5$ we cannot prove a modularity result in general. However, in Section \ref{sec:application} we will be able to prove modularity under some conditions. 

\section{\texorpdfstring{The conductor of $J_{2,r}$}{The conductor of the Jacobian}}
\label{Sec:conductor}
In this section, we compute the conductor of the compatible system of Galois representations $\{\rho_{J_{2,r},\lambda}\}$ in as much generality as possible.
 In order to simplify the exposition, we will leave the proof of some technical results involving cluster pictures for Section \ref{sec:odd-cond-general} (which can be ommitted without loss of continuity), but refer and use them in this section. For this purpose, let us consider the curve
\begin{equation}\label{eqn:generalfrey}
    C(z,s) : y^2 = F(x) = (-z)^{\frac{r-1}{2}} x h\left(2-\frac{x^2}{z}\right) +s,
\end{equation}
studied in \cite{computationconductor}. Note that the curve $C_{2,r}(a,b)$ defined in \eqref{eqn:frey} is a particular instance of $C(z,s)$ with $z = -ABb$ and $s = 2A^{(r+1)/2}B^{(r-1)/2}a$.

Following the notation in \cite[Notation 2.3]{computationconductor}, we define $\Delta=s^2-4z^r$, so for the curve $C_{2,r}(a,b)$,
\begin{equation}
    \Delta=4A^rB^{r-1}(Aa^2+Bb^r).
\end{equation}
From \cite[Lemma 3.1]{computationconductor}, the discriminant of $C(z,s)$ is given by 
\begin{equation}
    \Delta_{C(z,s)}= (-1)^{\frac{r-1}{2}} 2^{2(r-1)} r^r \Delta^{\frac{r-1}{2}},
\end{equation}
which is consistent with \eqref{eq:disc2rp}. To ease the notation, throughout this section we will write $C_{2,r}$ to denote the curve $C_{2,r}(a,b)$, defined over $K$. Since we are interested in Diophantine applications to equations of the form (\ref{eqn:main}), over this section we will assume that $A$ is squarefree and that $r\ge 5$ (since the case $r=3$ is well understood, see \cite[Chapter 5]{CarlosPhdThesis}).

We now provide some results characterizing the reduction type of $J_{2,r}$ over different primes of $K$. Given a rational prime $q$, let $\id{q}$ be a prime in $K$ above $q$.

\begin{prop} \label{prop:good-red-outside-2r}
If $q\nmid 2rAB(Aa^2+Bb^r)$ then $J_{2,r}$ has good reduction at $\id{q}$.
\end{prop}

\begin{proof}
If $q\nmid 2rAB(Aa^2+Bb^r)$ then $q$ does not divide the discriminant of $C_{2,r}$, given in (\ref{eq:disc2rp}). Then  the curve has good reduction at $\id{q}$, and so does its Jacobian.
\end{proof}

\begin{prop} \label{prop:mult-red-outside-2r}
If $q\nmid 2rAB$ and  $q\mid Aa^2 +Bb^r$ then both $C_{2,r}$ and $J_{2,r}$ have bad multiplicative reduction at $\id{q}$.
\end{prop}

\begin{proof}
By hypothesis, $q$ is an odd prime such that $q\nmid rs$ (note that $q \nmid a$ by the coprimality assumption). Moreover,  since $q\mid \Delta$, $\id{q}$ is a prime of bad reduction for $C_{2,r}$. Since $q \nmid A$, conditions (i) and (ii) in \cite[p.\ 7]{computationconductor} are satisfied and then the statement follows from \cite[Theorem 5.4]{computationconductor}.
\end{proof}

\begin{prop} \label{prop:Pot-Good-red}
If $q\neq 2$ and $q\mid AB$ then $C_{2,r}$  and $J_{2,r}$ have  potentially good reduction at $\id{q}$. Moreover, if $q\neq r$, the conductor exponent of $\{\rho_{J_{2,r},\lambda}\}$ at $\id{q}$ equals $2$.
\end{prop}

\begin{proof}
In this case we have $q\neq 2$, $q\mid \Delta$ and $q\mid s$. 

Applying \cite[Theorem 5.5]{BBB}, we can show that $C_{2,r}$ and $J_{2,r}$ have potentially good reduction at $\id{q}$ from the cluster picture of the curve. This cluster picture is computed in \cite[Corollary 3.5]{computationconductor} for the case $q\mid B$ and in Proposition \ref{prop:clusteroverK} for the case $q\mid A$.

The last statement follows from  \cite[Proposition 6.3]{computationconductor} and \cite[Theorem 5.4]{computationconductor} (resp.\  Proposition \ref{prop:newconductorK} in the present article) if $q\mid B$ (resp.\ $q\mid A$).
\end{proof}

Let $\id{r}$ be the unique prime ideal of $K$ above $r$, as in Section \ref{sec:notation}.

\begin{lemma}\label{lemma:pot-good-red-at-r}
If $r\nmid Aa^2+Bb^r$ then $C_{2,r}$ and $J_{2,r}$ have potentially good reduction at $\id{r}$.
\end{lemma}

\begin{proof}
If $r\mid AB$, the result follows from Proposition \ref{prop:Pot-Good-red}. Otherwise, $r\nmid \Delta$ and the results follow again by applying \cite[Theorem 5.5]{BBB} to the cluster picture given in \cite[Corollary 3.5]{computationconductor}.
\end{proof}

Determining the conductor exponent of the representation $\{\rho_{J_{2,r},\lambda}\}$ at primes dividing $2r$ is a significantly more difficult problem, and demands substantial effort. In those cases we have the following results.

\begin{theorem}\label{thm:conductor2rp}
Then the conductor exponent of $\{\rho_{J_{2,r},\lambda}\}$ at $\id{r}$ equals 
\begin{equation*}
    \mathfrak{n}_{\id{r}}(\rho_{J_{2,r},\lambda}) = 
    \begin{cases}
    2 &\text{ if } r \mid A, r\mid a \text{ and } r\not \equiv 7 \pmod{8}, \\
    0 &\text{ if } r \mid A, r\mid a \text{ and } r \equiv 7 \pmod{8}, \\
    \frac{r+5}{2} &\text{ if } r \mid A \text{ and } r\nmid a,  \\
    2 & \text{ if } r \nmid ABCc \text{ and } F(x) \text{ is reducible over } \Q_r, \\
    3 & \text{ if } r \nmid ABCc \text{ and } F(x) \text{ is irreducible over } \Q_r, \\
    r+2 & \text{ if } r \mid B \text{ and } r\nmid A, \\
    \frac{r+5}{2} & \text{ if } r \nmid ABc \text{ and } v_r(C) = 1, \\
    3 & \text{ if } r \nmid ABc \text{ and } v_r(C) = 2, \\
    2 & \text{ if } r \nmid AB \text{ and either } r \mid c \text{ or } v_r(C)\ge3.       
\end{cases}
\end{equation*}
\end{theorem}
\begin{proof}
    If $r\nmid A$, the result follows from \cite[Corollary 6.7]{computationconductor}. If $r\mid A$, the conductor exponent of $C_{2,r}$ at $\id{r}$ follows from \cite[Proposition 6.3]{computationconductor} and Proposition \ref{prop:newconductorK} in this article.
\end{proof}

Let $\id{q}_2$ be a prime of $K$ above 2. In this case, we are only able to find the conductor exponent  under certain restrictive hypotheses. This is the content of the following proposition, proved in \cite{ChenVillagra}.

\begin{prop}\label{prop:conductor-at-2}
Suppose that $v_2(Bb^r)\ge 6$. Then $C_{2,r}$ has potentially good reduction. Moreover, there exists an explicit quadratic character $\chi$ such that the twisted curve $C_{2,r}\otimes\chi$ acquires good reduction over any extension $L/K_{\id{q}_2}$ with ramification index $r$.

Furthermore, the conductor exponent of $\{\rho_{J_{2,r}\otimes\chi, \lambda}\}$ at $\id{q}_2$ equals 2.
\end{prop}

\begin{proof} 
The first statement follows by \cite[Corollary 4.8]{ChenVillagra}, setting $t=Aa^2/(Aa^2+Bb^r)$. The twist that appears in the statement is given by $\pm1$ or $\pm 2$, and it is explicitly determined in the proof of \cite[Proposition 4.3]{ChenVillagra}. 

The last statement follows by \cite[Proposition 2.16]{ChenVillagra}.
\end{proof}

One of the main contributions of the present article is the computation of the conductor exponent at odd primes dividing $A$. This was used in Proposition \ref{prop:mult-red-outside-2r} and Theorem \ref{thm:conductor2rp}, and it follows from Subsection \ref{sec:odd-cond-general} below.

\subsection{Odd conductor in general}\label{sec:odd-cond-general}

We keep the notation from the previous sections. The goal of this subsection is to extend the results of \cite{computationconductor} concerning the computation of the conductor of $C(z,s)$, previously established only for integers $z$ and $s$ satisfying conditions (i) and (ii) of Section 1.3 in loc.\ cit. We note that these results are no longer valid since, when specializing to 
\begin{equation}\label{eqn:valuesofzands}
    z = -ABb \quad\text{ and } \quad s = 2A^{(r+1)/2}B^{(r-1)/2}a,
\end{equation} 
condition~(i) fails to hold for primes $q$ dividing $A$.

Without loss of generality, we may assume that $A$ is square-free. Then, throughout this section we will focus on primes $q$ satisfying the following conditions: 
\begin{enumerate}[label=(\Roman*)]
    \item  \label{item1} $v_q(z) = 1$, 
    \item  \label{item2} $r = v_q(z^r) \le v_q(s^2) \iff v_q(s) \ge \frac{r+1}{2} > \frac{v_q(z^r)}{2}$.
\end{enumerate}
The main results of this section are the following propositions, which allow us to compute the odd conductor at primes satisfying hypotheses \ref{item1} and \ref{item2} above.

\begin{proposition}\label{prop:newconductorQ} 
Let $z, s\in\ZZ$ and let $C(z,s)/\Q$ be the curve defined in \eqref{eqn:generalfrey}. Let $q \in \Z$ be an odd prime of bad reduction. Suppose that $z,s$ and $q$ satisfy $\ref{item1}$ and $\ref{item2}$. Then, the conductor exponent of $C(z,s)$ at $q$ is given by Table \ref{table:conductorQ}. 
\begin{table}[H]
    \begin{tabular}{|c||l|}
    \hline
    $\mathfrak{n}(C/\Q_q)$ & \text{Condition} \\ \hline
    \hline
    ${r-1}$ &  ${q} \neq r$  \\ \hline
    $r-1$ & ${q} = {r}, \ v_r(s) \ge\frac{r+3}{2}, \ r \not \equiv 7 \pmod{8}$
    \\ \hline
    $r-3$ & ${q} = {r}, \ v_r(s) \ge\frac{r+3}{2}, \ r \equiv 7 \pmod{8}$ \\ \hline
    $\frac{3r-1}{2}$ & ${q} = {r}, \ v_r(s) = \frac{r+1}{2}$ \\ \hline
    \end{tabular}
    {\small
    \caption{Conductor exponent of $C(z,s):y^2 = F(x)$ at the rational prime $q$.}
        \label{table:conductorQ}}
\end{table}
\end{proposition}

Let $K=\Q(\zeta_r)^+$, as in Section \ref{sec:notation}.

\begin{proposition}\label{prop:newconductorK} 
Let $C(z,s)/K$ be the curve defined in \eqref{eqn:generalfrey}. Let $\id{q}$ be a prime of $K$ above an odd prime $q$. Suppose that $z,s$ and $q$ satisfy $\ref{item1}$ and $\ref{item2}$. Then, the conductor exponent of $C(z,s)$ at $\id{q}$ is given by Table \ref{table:conductorK}.
\begin{table}[H]
    \begin{tabular}{|c||l|}
    \hline
    $\mathfrak{n}(C/K_\id{q})$ & \text{Condition} \\ \hline \hline
    ${r-1}$ &  ${q} \neq r$  \\ \hline
    $r-1$ & ${q} = {r}, \ v_r(s) \ge\frac{r+3}{2}, \ r \not \equiv 7 \pmod{8}$
    \\ \hline
    $0$ & ${q} = {r}, \ v_r(s) \ge\frac{r+3}{2}, \ r \equiv 7 \pmod{8}$ \\ \hline
    $\frac{(r-1)(r+5)}{4}$ & ${q} = {r}, \ v_r(s) = \frac{r+1}{2}$ \\ \hline
    \end{tabular}
    {\small
    \caption{Conductor exponent of $C(z,s):y^2 = F(x)$ at the prime $\id{q}$ of $K$.}
        \label{table:conductorK}}
\end{table}
\end{proposition}

\subsubsection{Cluster pictures} 

The conductor of hyperelliptic curves can be deduced from their cluster picture. With the aim of proving Propositions \ref{prop:newconductorQ} and \ref{prop:newconductorK} we will construct the cluster picture of $C(z,s)$.

Let $q \ge 3$ be an odd prime. Then, we will denote by $v_q(\cdot)$ the standard $q$-adic valuation of $\Q_q$. When considering a field extension $K/\Q_q$ with ramification degree $e_{K/\Q_q}$, we will also write $v_q(\cdot)$ to denote the unique extension of the valuation of $\Q_q$ to $K$ satisfying that $v_q(q) = 1$. Finally, if $\id{q}$ is the unique prime ideal of the ring of integers of $K$ above $q$, $v_{\id{q}}(\cdot)$ will denote the valuation of $K$ normalized so that $v_\id{q}(q) = e_{K/\Q_q}$. Note that, by our choice of notation, $v_{\id{q}}(x) = e_{K/\Q_q}v_{q}(x)$ for all $x \in K$.

If $z$ and $s$ are the parameters of $C(z,s)$, we let
\begin{equation}\label{eqn:F}
    F(x)=(-z)^\frac{r-1}{2}xh\left(2-\frac{x^2}{z}\right)+s,
\end{equation}
be the defining polynomial of $C(z,s)$. Set $\mathcal{R}=\{\gamma_0,\cdots,\gamma_{r-1}\}\subseteq \overline{\Q}_q$ to be its set of roots. By \cite[Proposition 2.2]{computationconductor}, the roots of $F(x)$ are of the form $\gamma_j=\zeta_r^j\alpha_0+\zeta_r^{-j}\beta_0$, where 
\begin{equation*}\label{eqn:alpha0beta0}
 	\alpha_0=-\sqrt[r]{\frac{s+\sqrt{\Delta}}{2}}\in\overline{\Q}_q, \quad \quad \beta_0=-\sqrt[r]{\frac{s-\sqrt{\Delta}}{2}}\in\overline{\Q}_q.
\end{equation*}

Recall that, following \cite{computationconductor}, we use the short-hand notation where $\Delta$ is given by the expression $s^2-4z^r$. Hence we have 
\begin{align}\label{eq:alpha-beta}
    \alpha_0^r-\beta_0^r=-\sqrt{\Delta} \quad \text{and} \quad \alpha_0^r+\beta_0^r=-s.
\end{align}

\begin{lemma}\label{lemma:valuations}
Let $q \in \Z$ be an odd prime dividing $z$ and $s$ and such that \ref{item1} and \ref{item2} are satisfied. Then, for all $0 \le j \le r-1$,
\begin{equation}
    v_q(\alpha_0) = v_q(\beta_0) = v_q(\alpha_0-\zeta_r^j\beta_0) = \frac{1}{2} \  \text{ and } \ v_q(\sqrt{\Delta}) = \frac{r}{2}.
\end{equation}
\end{lemma}

\begin{proof}

The last claim follows directly from \ref{item2} and the definition of $\Delta$. Adding and substracting the expressions for $\alpha_0^r - \beta_0^r$ and $\alpha_0^r + \beta_0^r$ in \eqref{eq:alpha-beta}, we find that
\begin{equation}
    v_q(\alpha_0) = v_q(\beta_0) = \frac{1}{r}\min\{v_q(\sqrt{\Delta}), v_q(s)\} = \frac{1}{2},
\end{equation}
where we have used \ref{item2}. Hence, it follows that $v_q(\alpha_0-\zeta_r^j\beta_0)\ge \min\{v_q(\alpha_0), v_q(\beta_0)\}\ge1/2$ for all $0\le j\le r-1$. Since 
\begin{equation}
    \frac{r}{2}=v_q\left(-\sqrt{\Delta}\right) = v_q(\alpha_0^r-\beta_0^r)=\sum_{j=0}^{r-1}v_q(\alpha_0-\zeta_r^j\beta_0)\ge r\cdot\frac{1}{2}=\frac{r}{2},
\end{equation}
all inequalities are in fact equalities, proving that $v_q(\alpha_0) = v_q(\beta_0) = v_q(\alpha_0-\zeta_r^j\beta_0) = 1/2$.
\end{proof}

With the help of the previous lemma, it is straightforward to compute the cluster pictures of $C(z,s)$, both over $\Q$ and over $K$. For this, we introduce some terminology related to cluster pictures.

\begin{definition} 
A \textit{cluster} $\s$ is a non-empty subset of $\mathcal{R}$ of the form $\s=D\cap \mathcal{R}$ for some disc $D=\{x\in\overline{K} : v(x-z)\ge d\}$, where $z\in \overline{K}$ and $d\in\Q$.
\end{definition}

It is well known that given two clusters $\s_1$ and $\s_2$, then either they are disjoint or one is contained in the other. If $\s,\s'$ are two clusters such that $\s'\subsetneq \s$ is a maximal subcluster, we refer to $\s'$ as a \textit{child} of $\s$. Consistently, $\s$ will be the \textit{parent} of $\s'$, and we will denote this relation by $\s = P(\s')$.

We denote by $|\s|$ the number of elements of $\mathcal{R}$ belonging to $\s$. We will abuse the language by saying that $\s$ is \textit{odd} (resp.\ \textit{even}) if $|\s|$ is odd (resp.\ even).  Also, if $|\s|>1$ we  say that $\s$ is \textit{non-trivial}, and in such a case we can define its \textit{depth} $d_\s$ to be
\begin{equation}
    d_\s=\min \{v(r-r') : r,r'\in \s\}.
\end{equation}

\begin{notation}
In the cluster picture of $C(z,s)$, we denote by \smash{\raise4pt\hbox{\clusterpicture\Root[A]{1}{first}{r1};\endclusterpicture}} each element of $\mathcal{R}$, while clusters are represented by an oval, along with its depth, as the following example for $r=3$: 
\begin{equation*}
    \clusterpicture           
    \Root[A] {1} {first} {r1};
    \Root[A] {1} {r1} {r2};
    \Root[A] {1} {r2} {r3};
    \ClusterLDName c[][d_\mathcal{R}][] = (r1)(r2)(r3);
    \endclusterpicture
\end{equation*}
\end{notation}

\begin{proposition}\label{prop:clustersQ}
Let $q$ be an odd prime of bad reduction for $C/\mathbb{Q}_q$ dividing $z$ and $s$ and satisfying the hypotheses \ref{item1} and \ref{item2}. Then, the cluster picture of $C/\Q_q$ is
\begin{equation*}
    \clusterpicture            
    \Root[A] {1} {first} {r1};
    \Root[A] {1} {r1} {r2};
    \Root[A] {1} {r2} {r3};
    \Root[Dot] {3} {r3} {r4};
    \Root[Dot] {} {r4} {r5};
    \Root[Dot] {} {r5} {r6};
    \Root[A] {3} {r6} {r7};
    \Root[A] {1} {r7} {r8};
    \Root[A] {1} {r8} {r9};
    \ClusterLDName c[][d_\mathcal{R}][] = (r1)(r2)(r3)(r4)(r5)(r6)(r7)(r8)(r9);
    \endclusterpicture, 
    \text{ where }d_{\mathcal{R}}=\begin{cases}
        \frac{1}{2} & \text{ if } q\neq r,\\
        \frac{1}{r-1}+\frac{1}{2} & \text{ if } q= r,
    \end{cases}
\end{equation*}
\end{proposition}

\begin{proof}  
By \cite[Lemma 2.5]{computationconductor} and Lemma \ref{lemma:valuations}, we have that for all ${0\le k < j \le r-1}$,
\begin{equation}
    v_q(\gamma_k-\gamma_j)=v_q(\zeta_r^k(1-\zeta_r^{j-k})(\alpha_0-\zeta_r^{-j-k}\beta_0))=
    \begin{cases}
    \frac{1}{2} & \text{ if } q\neq r,\\
    \frac{1}{r-1} + \frac{1}{2} & \text{ if } q=r,
    \end{cases}
\end{equation}
giving the desired cluster.
\end{proof}

\begin{proposition}\label{prop:clusteroverK}
Let $q$ be an odd rational prime, let $\id{q}$ be a prime of $K$ above $q$ of bad reduction for $C/K$, and let $\id{r}$ be the unique prime ideal of $K$ over $r$. Then, the cluster picture of $C/{K_\id{q}}$ is
\begin{equation*}
    \clusterpicture            
    \Root[A] {1} {first} {r1};
    \Root[A] {1} {r1} {r2};
        \Root[A] {1} {r2} {r3};
    \Root[Dot] {3} {r3} {r4};
    \Root[Dot] {} {r4} {r5};
    \Root[Dot] {} {r5} {r6};
    \Root[A] {3} {r6} {r7};
        \Root[A] {1} {r7} {r8};
        \Root[A] {1} {r8} {r9};
    \ClusterLDName c[][d_\mathcal{R}][] = (r1)(r2)(r3)(r4)(r5)(r6)(r7)(r8)(r9);
    \endclusterpicture, 
    \text{ where }
    d_{\mathcal{R}}=
    \begin{cases}
        \frac{1}{2} & \text{ if } q\neq r,\\
        \frac{r-1}{4}+\frac{1}{2} & \text{ if } q= r,
    \end{cases}
\end{equation*}
\end{proposition}

\begin{proof}
Since $K$ is a subfield of $\Q(\zeta_r)$, it is unramified at any prime $q \neq r$ and totally ramified at $q = r$. Since the degree of the extension $K/\Q$ is ${(r-1)}/{2}$, it follows that, for any $a \in K$, 
\begin{align}\label{eq:idealvaluation}
v_{\id{q}}(a)=
\begin{cases}
v_q(a) & \text{ if } \id{q} \neq \id{r}, \\
\frac{r-1}{2}v_r(a) & \text{ if } \id{q} = \id{r}.
\end{cases} 
\end{align}
Then, the desired result is a direct consequence of Proposition \ref{prop:clustersQ}.
\end{proof}

\subsubsection{A criterion for irreducibility}
Once we have the cluster pictures, the next step is to find a criterion for determining whether $F(x)$ is irreducible, since the one given in \cite[Lemma 4.4]{computationconductor} is no longer valid. This is the content of the following lemma.

\begin{lemma}\label{lemma:irreducibility}
Let $q$ be an odd prime dividing $z$ and $s$ and satisfying \ref{item1} and \ref{item2}. Then, $F(x)$ is irreducible over $\Q_q$ if, and only if, $q = r$ and $v_r(s) = (r+1)/2$.
\end{lemma}

\begin{proof}
Let us first consider the case $q=r$. Recall that, from \ref{item2}, we have $v_r(s)\ge (r+1)/2$. 
    
Suppose first that $v_r(s) \ge (r+3)/2$. In order to show that $F(x)$ is reducible, it suffices to show that the Newton polygon of $F(x)$ has at least two segments, by \cite[Chapter 6, Theorem 3.1]{Cassels}. Note that 
\begin{equation}
    F(x) = x^r + \dots \pm rz^{\frac{r-1}{2}}x + s.
\end{equation}
Consider the three points $P_0 = (0,v_r(s))$, $P_1 = (1, v_r(rz^{\frac{r-1}{2}})) = (1, \frac{r+1}{2})$ and $P_r = (r, 0)$. Let $m_1$ denote the slope of the segment connecting $P_0$ and $P_1$ and $m_r$ the slope of the segment connecting $P_0$ and $P_r$. It is elementary to see that 
\begin{equation}\label{eqn:slopes}
    m_1 = \frac{r+1}{2}-v_r(s) \quad \text{and} \quad m_r = -\frac{v_r(s)}{r}.
\end{equation}
It is clear that if $m_1 < m_r$, the resulting Newton polygon will have at least two edges. It is immediate to see that 
\begin{equation}\label{eqn:inequality}
    m_1 < m_r \quad \text{ if and only if } \quad  v_r(s) > \frac{r(r+1)}{2(r-1)}.
\end{equation}
Note that we had assumed that $v_r(s) \ge (r+3)/2$. Now, we note that 
\begin{equation}
    \frac{r+3}{2} > \frac{r(r+1)}{2(r-1)} \quad \text{ if and only if } \quad r > 3,
\end{equation}
and so \eqref{eqn:inequality} holds. Consequently, $F(x)$ is reducible.

Secondly, assume that $q=r$ and that $v_r(s) = (r+1)/2$. By \cite[Proposition 4.1]{computationconductor}, which still applies to our case, $F(x)$ is reducible if, and only if, it has a root $\gamma_{k_0} \in \Q_r$. If this were the case, we note that 
\begin{equation}
    0\equiv F(\gamma_{k_0}) \equiv \gamma_{k_0}^r \pmod{r},
\end{equation}
so that $v_r(\gamma_{k_0}) \ge 1$. Now, let us write $F(x)$ in the following manner:
\begin{equation}
    F(x) = x^r + a_{r-2}zx^{r-2} + a_{r-4}z^2x^{r-4} + \dots \pm rz^\frac{r-1}{2}x + s.
\end{equation}
Since $v_r(\gamma_{k_0}) \ge 1$ and $v_r(r) = v_r(z) = 1$, we have that
\begin{equation}
    F(\gamma_{k_0}) \equiv s \not \equiv 0 \pmod{r^{\frac{r+3}{2}}},
\end{equation}
where the last congruence follows since $v_r(s) = (r+1)/2$. Thus, $F(x)$ has no root over $\Q_r$ and is therefore irreducible.

If $q \neq r$, we consider $P_1 = (1, \frac{r-1}{2})$ and we adapt \eqref{eqn:slopes} in the following way:
\begin{equation}
    m_1 = \frac{r-1}{2}-v_q(s) \quad \text{and} \quad m_r = -\frac{v_q(s)}{r}.
\end{equation}
Then, we show that 
\begin{equation}
    m_1 < m_r \quad \text{ if and only if } \quad v_q(s) > \frac{r}{2}, 
\end{equation}
which is true by \ref{item2}. Consequently, $F(x)$ is reducible, as desired.
\end{proof}

\subsubsection{Index of ramification}
Let $L$ be the splitting field of $F(x)$. As the final ingredient for computing the conductor of $C(z,s)$, we need to compute the ramification index of $L/\Q_q$. For this, we need the following lemma.

\begin{lemma}\label{lemma:auxramification}
Let $q,z,s \in \ZZ$ with $q$ prime and satisfying \ref{item1}, and such that $v_r(s) = (r+1)/2$. Then, the extension $\Q_r(\gamma_i)/\Q_r$ is totally ramified of degree $r$ for any root $\gamma_i$ of $F(x)$.
\end{lemma}

\begin{proof}
If $v_r(s) = (r+1)/2$, the polynomial $F(x)$ is irreducible over $\Q_r$ by Lemma \ref{lemma:irreducibility}, and so the extension $\Q_r(\gamma_i)/\Q_r$ has degree $r$. By an identical argument to the proof of Lemma \ref{lemma:irreducibility}, we then show that $0 < v_r(\gamma_i) < 1$. This is only possible if the extension $\Q_r(\gamma_i)/\Q_r$ is ramified. Since $r$ is a prime, it follows that the extension is necessarily totally ramified.
\end{proof}

The following is an adaptation of  \cite[Lemma 4.10]{computationconductor}. As in Section \ref{sec:notation}, let $\omega = \zeta_r + \zeta_r^{-1}$ and  set $\tau = \zeta_r-\zeta_r^{-1}$. Denote $E:=\Q_q(\omega,\tau\sqrt{\Delta}) =\Q_q(\tau\sqrt{\Delta})$, where the equality holds by \cite[Remark 4.8]{computationconductor}.

\begin{lemma} \label{lemma:indexram}
Let $q$ be an odd prime and let $z,s \in \ZZ$ be integers satisfying \ref{item1} and \ref{item2}. The ramification index  of $E/\Q_q(\omega)$ equals
\begin{equation}
    e_{E/\Q_q(\omega)}=
    \begin{cases}
        1 & \text{if } q=r \text{ and } r\equiv 3\pmod 4,\\
        2 & \text{otherwise}.
    \end{cases}
\end{equation}
\end{lemma}

\begin{proof}
The minimal polynomial of $E$ over $\Q_q(\omega)$ divides
\begin{equation}\label{eqn:definingpoly}
    x^2 - \Delta\tau^2 = x^2 + \Delta(4 - \omega^2) \in \Q_q(\omega)[x].
\end{equation}
In particular, this implies that $e_{E/\Q_q(\omega)}\le 2$. Moreover, by \cite[Theorem 6.1]{Cassels},  $E/\Q_q(\omega)$ will be unramified if and only if $v_\mathfrak{q}(-4\Delta(\omega^2-4))$ is even, where $\id{q}$ is the unique prime in $\Q_q(\omega)$ above $q$, and $v_\id{q}$ its respective valuation. Since  $v_q(\Delta) = r$ (see Lemma \ref{lemma:valuations}), the result follows from (\ref{eq:idealvaluation}) and the fact that $v_\id{q}(\omega^2-4)=\delta_{q,r}$, where $\delta$ is the Kronecker delta.
\end{proof}

\begin{proposition}
    \label{prop:ramification}
The ramification index $e_{L/\Q_q}$ of $L/\Q_q$ is equal to 
\begin{equation}\label{eq:ramindex}
    e_{L/\Q_q}=\varepsilon_1\varepsilon_2e_{E/\Q_q(\omega)},
\end{equation}
where
\begin{equation*}
    \varepsilon_1=\begin{cases}
    \frac{r-1}{2} & \text{if } q=r,\\
    1  & \text{if } q\neq r
\end{cases}
\quad \text{ and } \quad 
\varepsilon_2=\begin{cases}
r & \text{if } q=r \text{ and } v_r(s) = \frac{r+1}{2},\\
1  & \text{otherwise.}
\end{cases}
\end{equation*}
\end{proposition}

\begin{proof}
We claim that $\varepsilon_1:=e_{\Q_q(\omega)/\Q_q}$ and $\varepsilon_2:=e_{L/E}$. Then, it is clear by the tower law that (\ref{eq:ramindex}) holds. To prove the formula for the values of $\varepsilon_1$ note that the extension $\Q_q(\omega)/\Q_q$ is unramified (resp.\ totally ramified of degree $(r-1)/2$) if $q\neq r$ (resp.\ $q=r$).

Suppose $F(x)$ is reducible over $\Q_q$.  By \cite[Proposition 4.1]{computationconductor}, there exists $0\le k_0 \le r-1$ such that $\gamma_{k_0}\in\Q_q$. If $\zeta_r\notin\Q_q$ then $k_0=0$, by \cite[Remark 4.2]{computationconductor}. Then, \cite[Proposition 4.7]{computationconductor} implies that $L=E$. On the other hand, if $\zeta_r\in\Q_q$  we also have that $L=E$ by \cite[Proposition 4.7]{computationconductor}, taking $j=k_0$ in the statement. In either case, $e_{L/E} = 1 = \varepsilon_2$.

Therefore,  it only remains to consider the case where $F(x)$ is irreducible, which, by Lemma \ref{lemma:irreducibility}, corresponds to $q = r$ and $v_r(s) = (r+1)/2$. We then proceed to compute $e_{L/E}$. In this case, by \cite[Corollary 4.11]{computationconductor} we have that $e_{\Q_q(\alpha_0)/\Q_q(\alpha_0^r)}=r$.
Using the chain of field extensions
\begin{equation}
    \Q_q\subseteq\Q_q(\alpha_0^r)\subseteq\Q_q(\alpha_0)
    \subseteq\Q_q(\zeta_r,\alpha_0),
\end{equation}
we get that $r$ divides the ramification index of the extension $\Q_q(\zeta_r,\alpha_0)/\Q_q$. On the other hand, using the chain of field extensions 
\begin{equation}
    \Q_q\subseteq E\subseteq L\subseteq \Q_q(\zeta_r,\alpha_0),
\end{equation}
we get that $r\mid e_{L/E}$, since $r$ is prime and it neither divides the ramification index of $E/\Q_q$ (which is less than or equal to $r-1$) nor that of $\Q_q(\zeta_r,\alpha_0)/L$ (being less than or equal to $2$). Therefore, since $e_{L/E}\le r$, we get that $e_{L/E}=r$. The value of $e_{L/E}$ is then equal in all cases to the value of $\varepsilon_2$ given in the statement of this proposition.
\end{proof}

We recall that our aim in this section is to compute the conductor of the curve. For convenience, we recall the following fact, which is \cite[Lemma 1.6]{computationconductor} (also \cite[Definition 2.2 and Proposition 2.11]{CelineClusters}).

\begin{lemma}\label{lemma:twoparts}
Let $p$ be a prime number, $K$ a local field with $\Q_p \subseteq K$ and let $C/K$ be a hyperelliptic curve of genus $g$ with defining polynomial $f(x)$. Then, the conductor exponent $\mathfrak{n}(C/K)$ decomposes as
\begin{equation}
    \mathfrak{n}(C/K) = \mathfrak{n}_{wild}(C/K) + \mathfrak{n}_{tame}(C/K),
\end{equation}
where $\mathfrak{n}_{wild}(C/K)$ and $\mathfrak{n}_{tame}(C/K)$ are called, respectively, the wild and the tame part of the conductor. In addition, if $L$ is the splitting field of $f(x)$, we have that $\mathfrak{n}_{wild}(C/K) = 0$ if $\gcd(e_{L/K}, p) = 1$.
\end{lemma}

 By Proposition \ref{prop:ramification}, $\gcd(e_{L/\Q_q}, r) = 1$ unless $q=r$ and $v_r(s) = (r+1)/2$. By Lemma \ref{lemma:twoparts}, this will be the most difficult case going forward, since we need to account for the wild part of the conductor.

\subsubsection{Wild conductor}
In this section, we suppose that $q=r$, $v_r(s) = (r+1)/2$ and $v_r(z)=1$, so that we are in the case where $F(x)$ is irreducible, where the wild part of the conductor exponent is non zero. In order to simplify notation, we let $F = \Q_r(\alpha_0^r) = \Q_r(\sqrt{\Delta})$ and $M = \Q_r(\alpha_0)$. 

Finally, we note that, by Lemma \ref{lemma:valuations}, the element 
\begin{equation}
    \pi := \frac{\sqrt{\Delta}}{r^{\frac{r-1}{2}}},
\end{equation}
is a uniformizer of $\cO_F$, since $v_r(\pi) = 1/2 = 1/e_{F/\Q_r}$. Let $v_\pi$ be its associated valuation.

\begin{lemma}\label{lemma:auxcongruence}
Let $\pi$ be as above. Then
\begin{equation}
    \left(\frac{\sqrt{\Delta}}{2\pi^r}\right)^r \equiv \frac{\sqrt{\Delta}}{2\pi^r} \pmod{\pi^2}.
\end{equation}
\end{lemma}

\begin{proof}
By direct computation, we have that 
\begin{equation}
    \frac{\sqrt{\Delta}}{2\pi^r} = \frac{r^{\frac{r(r-1)}{2}}}{2 \Delta^\frac{r-1}{2}} \in \Q_r.
\end{equation}
By Fermat's Little Theorem, we get that 
\begin{equation}
    \left(\frac{\sqrt{\Delta}}{2\pi^r}\right)^r \equiv \frac{\sqrt{\Delta}}{2\pi^r} \pmod{r}.
\end{equation}
The result then follows since, as the extension $F/\Q_r$ is totally ramified and quadratic, we have that $v_\pi(r) = 2$.
\end{proof}

We define the element $u$ by the expression
\begin{equation}\label{eqn:definitionu}
    u := \frac{\alpha_0^r}{\pi^r} = -\frac{s+\sqrt{\Delta}}{2\pi^r}.
\end{equation}

\begin{lemma}\label{lemma:uniformiser}
Let $u$ be defined as in \eqref{eqn:definitionu}. Then
\begin{equation}
    v_\pi(u^r-u) = 1.
\end{equation}
\end{lemma}

\begin{proof}
Since $F/\Q_r$ is totally ramified, $\left(\cO_F/\pi\cO_F\right)^\times$ has order $r-1$. Then 
\begin{equation}
    u^{r-1} \equiv 1 \pmod{\pi},
\end{equation}
and so $v_\pi(u^r-u) \ge 1$. Now, by the binomial expansion, we have that 
\begin{equation}\label{eqn:ur}
    u^r = -\frac{(s+\sqrt{\Delta})^r}{(2\pi^r)^r} = - \frac{\sum_{i=0}^r \binom{r}{i}s^i\sqrt{\Delta}^{r-i}}{(2\pi^r)^r}.
\end{equation}
Since $v_\pi(\cdot) = 2v_r(\cdot)$, we have that $v_\pi(\sqrt{\Delta}) = r$, $v_\pi(s) = r+1$ and $v_\pi\left(\binom{r}{i}\right) \ge 2$  for all $1 \le i \le r-1$. By \eqref{eqn:ur}, this means that 
\begin{equation}\label{eqn:auxcongruence1}
    u^r \equiv -\frac{\sqrt{\Delta}^r}{2^r\pi^{r^2}} \equiv - \frac{\sqrt{\Delta}}{2\pi^r} \pmod{\pi^2},
\end{equation}
where the last congruence follows by Lemma \ref{lemma:auxcongruence}. On the other hand, we have that 
\begin{equation}\label{eqn:auxcongruence2}
    u \equiv - \frac{s+\sqrt{\Delta}}{2\pi^r} \pmod{\pi^2}.
\end{equation}
Assume for contradiction that $u^r \equiv u \pmod{\pi^2}$. Then, combining \eqref{eqn:auxcongruence1} and \eqref{eqn:auxcongruence2}, we get that $v_\pi(s) \ge r+2$, which is a contradiction with the fact that $v_\pi(s) = r+1$. Consequently, the lemma follows.
\end{proof}

\begin{lemma}\label{lemma:minimalpoly}
Keeping the previous notation,
\begin{equation}
    M := \Q_r(\alpha_0) = \Q_r(u^{\frac{1}{r}}, \pi).
\end{equation}
Consequently, the minimal polynomial of the extension $M/F$ is $p(x) = x^r-u$.
\end{lemma}

\begin{proof}
Firstly, we note that $\pi \in F \subseteq M$ and $u^{\frac{1}{r}} = \frac{\alpha_0}{\pi} \in M$, and so $\Q_r(u^{\frac{1}{r}}, \pi) \subseteq M$. Conversely, $\alpha_0 = u^{1/r}\pi \in \Q_r(u^{\frac{1}{r}}, \pi)$ and so $M \subseteq \Q_r(u^{\frac{1}{r}}, \pi)$.

The statement about the polynomial corresponds to the fact that $F=\Q_r(\pi)$ and $M = F(u^{1/r})$ so that the minimal polynomial of $M/F$ is precisely $x^r-u$.
\end{proof}

\begin{lemma}\label{lemma:eisenstein}
Let $p(x)$ be as defined in Lemma \ref{lemma:minimalpoly}. Then $q(x) = p(x+u) \in F[x]$ is an Eisenstein polynomial.
\end{lemma}

\begin{proof}
We note that 
\begin{equation}
    q(x) = (x+u)^r-u = x^r + \sum_{i=1}^{r-1} \binom{r}{i} u^ix^{r-i} + u^r-u.
\end{equation}
The result then follows by Lemma \ref{lemma:uniformiser}, since $v_\pi(u^r-u) =1$ and $v_\pi\left(\binom{r}{i}\right) = 2$ for all $1 \le i \le r-1$.
\end{proof}

\begin{lemma}\label{lemma:valuationdiscriminant}
Let $u$ be defined as in \eqref{eqn:definitionu}. Then $u^{1/r}-u$ is a uniformiser of $M/F$ and $v_\pi\left(\Delta\left(M/F\right)\right) = 2r$.
\end{lemma}

\begin{proof}
The first statement follows from Lemma \ref{lemma:eisenstein}, since $u^{1/r}-u$ is a root of $q(x)$, which is an Eisenstein polynomial. The second statement follows similarly to \cite[Lemma 4.10]{computationconductor}, since the discriminant of $p(x)$ is 
\begin{equation}
    \Delta(p(x)) = (-1)^{r\frac{r-1}{2}}r^ru^{r-1}.
\end{equation}
Then, $\cO_M = \cO_F[u^{1/r}-u] = \cO_F[u^{1/r}]$ and the discriminant $\Delta(M/F)$ is, up to units, equal to $\Delta(p(x))$.
\end{proof}

With this, we may compute the valuation of the discriminant, which will allow us to compute the wild conductor.

\begin{proposition}\label{prop:valuationdiscriminant}
Keeping the previous notation, we have 
\begin{equation}
    v_r(\Delta\left(\Q_r(\gamma_0)/\Q_r\right)) = \frac{3r-1}{2}.
\end{equation}
\end{proposition}

\begin{proof}
Follows from Lemma \ref{lemma:valuationdiscriminant}, by an identical argument to the one used in \cite[Lemma 4.10 (3)]{computationconductor}.
\end{proof}

\begin{proposition}\label{prop:wildconductor}
Let $s,z \in \Z$ and suppose that \ref{item1} and \ref{item2} are satisfied for $q=r$ and that $v_r(s) = (r+1)/2$. Then, the wild conductor exponent of $C(z,s)$ over $\Q_r$ is 
\begin{equation}
    \id{n}_{wild}(C(z,s)/\Q_r) = \frac{r+1}{2},
\end{equation}
and the wild conductor exponent of $C(z,s)$ over $K_\id{r}$ is 
\begin{equation}
    \id{n}_{wild}(C(z,s)/K_\id{r}) = \frac{r^2-1}{4}.
\end{equation}
\end{proposition}

\begin{proof} 
Since $F(x)$ is irreducible and has degree $r$, a particular case of \cite[Theorem 12.3]{BBB} (see e.g.\ \cite[Theorem 1.12]{computationconductor}), implies that
\begin{equation}
    \mathfrak{n}_{\text{wild}}(C/\Q_r)=v_r(\Delta({\Q_r(\gamma_0)/\Q_r})) - r + f_{\Q_r(\gamma_0)/\Q_r},
\end{equation}
where $f_{\Q_r(\gamma_0)/\Q_r}$ denotes the residual degree of $\Q_r(\gamma_0)/\Q_r$. Since the ramification degree of $\Q_r(\gamma_0)/\Q_r$ is $r$ by Lemma \ref{lemma:auxramification}, $f_{\Q_r(\gamma_0)/\Q_r} = 1$. The desired result over $\Q_r$ then follows from Proposition \ref{prop:valuationdiscriminant}.

Finally, the statement for the wild conductor over $K_\id{r}$ follows by applying \cite[Lemma 2.2]{CelineClusters} since the ramification degree of $K/\Q_r$ is $(r-1)/2$.
\end{proof}

\subsubsection{Proof of Propositions \ref{prop:newconductorQ} and \ref{prop:newconductorK}} We now have all the tools to conclude the computation of the conductor of $C(z,s)$ and, therefore, to finish the proof of Propositions \ref{prop:newconductorQ} and \ref{prop:newconductorK}. First, we need to introduce more notation within the cluster picture context. 
\begin{definition}\label{def:twin}
An even cluster whose children are all even is called \textit{übereven}.
\end{definition}

\begin{definition} 
For two clusters $\s_1$ and $\s_2$ we write $\s_1 \wedge \s_2$ for the smallest cluster containing both of them.
\end{definition}

Given a cluster $\s$, we denote by $\tilde{\s}$ the set of odd children of $\s$ and consider 
\begin{equation}
    \tilde{\lambda}_\s :=\frac{1}{2}\left(|\tilde{\s}|d_\s + \sum_{\gamma\notin \s} d_{\{\gamma\}\wedge\s}\right).
\end{equation}

Let $L$ be as in the preceding subsections and let $I_K$ be the inertia group of $\Gal(L/K)$. If $A$ is any set where $I_K$ acts on, we let $A/I_K$ be the set of orbits of $A$ under the action of $I_K$. Furthermore, if $\s$ is a cluster, we denote by $I_\s$ the stabiliser of $\s$ under $I_K$. If $a\in\Q$, we define $\xi_\s(a)=\max\{-v_2([I_K:I_\s]a),0\}$, where $v_2$ is the $2$-adic valuation of $\Q$. Finally, define the following sets attached to the cluster picture of $C(z,s)$:
\begin{align*}
    & U = \{\s : \s\neq\mathcal{R} \text{ is odd and } \xi_{P(\s)}(\tilde{\lambda}_{P(\s)})\le \xi_{P(\s)}(d_{P(\s)})\},\\
    & V = \{\s : \s \text{ is proper, non-übereven and } \xi_\s(\tilde{\lambda}_\s)=0\}.
\end{align*}

\begin{proof}[Proof of Proposition \ref{prop:newconductorQ}]

If $q \neq r$, by Proposition \ref{prop:clustersQ}, we have that the only odd clusters are the singletons $\{\gamma_i\}$, which have $P(\{\gamma_i\}) = \mathcal{R}$. We therefore compute 
\begin{equation}
    \tilde{\lambda}_\mathcal{R} = \frac{r}{4},\quad \xi_\mathcal{R}(\tilde{\lambda}_\mathcal{R}) = 2, \quad \text{and} \quad \xi_\mathcal{R}(d_\mathcal{R}) = 1,
\end{equation}
and consequently, it follows that $U = \emptyset$. Similarly, the only proper non-übereven cluster is $\mathcal{R}$ and so $V = \emptyset$. By \cite[Theorem   1.12]{computationconductor}, it follows that 
\begin{equation}
    \id{n}_{tame}(C(z,s)/\Q_q) = r-1.
\end{equation}
By Proposition \ref{prop:ramification} and Lemma \ref{lemma:twoparts}, the wild part of the conductor is $0$ and, consequently, $\id{n}(C(z,s)/\Q_q) = \id{n}_{tame}(C(z,s)/\Q_q)$. Now, assume that $q=r$. In this case, a similar argument applies and we find that 
\begin{equation}
    d_\mathcal{R} = \frac{r+1}{2(r-1)} \quad \text{and} \quad \tilde{\lambda}_\mathcal{R} = \frac{r(r+1)}{4(r-1)}.
\end{equation}
Clearly, $0\neq\xi_\mathcal{R}(\tilde{\lambda}_\mathcal{R}) > \xi_\mathcal{R}(d_\mathcal{R})$ unless $r \equiv 7 \pmod{8}$. Consequently, if $r \not \equiv 7 \pmod{8}$, it follows that $U=V=\emptyset$ and 
\begin{equation}
    \id{n}_{tame}(C(z,s)/\Q_q) = r-1.
\end{equation}
If $r \equiv 7 \pmod{8}$, we have that 
\begin{equation}
    U = \{\{\gamma_i\} \mid i = 0,1,\dots, r-1\} \quad \text{and} \quad V = \{\mathcal{R}\}.
\end{equation}

We now distinguish two cases depending on whether $v_r(s) \ge (r+3)/2$ or $v_r(s) = (r+1)/2$. In the first case, the polynomial $F(x)$ is reducible and, by \cite[Proposition 4.1]{computationconductor}, there exists some $k_0$ with $0 \le k_0 \le r-1$ with $\gamma_{k_0} \in \Q_r$. In addition, Proposition \ref{prop:ramification} gives that the ramification degree of $E/\Q_r$ is $(r-1)/2$. Therefore, there are precisely two distinct orbits of inertia in the set $\{\gamma_i \mid i \neq k_0\}$, which along with the singleton orbit $\{\gamma_{k_0}\}$ give rise to three orbits in $U$ under the action of inertia. Therefore, the conductor in this situation is 
\begin{equation}
    \id{n}(C(z,s)/\Q_r) = r-1-3+1 = r-3.
\end{equation}
Now, let us suppose that $v_r(s) = (r-1)/2$. By Lemma \ref{lemma:auxramification}, the ramification degree of $\Q_r(\gamma_0)/\Q_r$ is $r$ and, therefore, there is a unique orbit of inertia in the set $\{\gamma_0, \dots, \gamma_{r-1}\}$. Therefore, the tame part of the conductor is equal to 
\begin{equation}
    \id{n}_{tame}(C/\Q_r) = r-1-1+1 = r-1.
\end{equation}
Note that this is the same tame conductor that we obtained if $r \not \equiv 7 \pmod{8}$ and the result is therefore independent of the congruence class of $r$ modulo $8$ whenever $F(x)$ is irreducible.

Adding up the wild part of the conductor given in Proposition \ref{prop:wildconductor}, we get that, in this case
\begin{equation}
    \id{n}(C/\Q_r) = r-1+ \frac{r+1}{2} = \frac{3r-1}{2},
\end{equation}
as desired.
\end{proof}

The proof of Proposition \ref{prop:newconductorK} follows a similar line, and we shall omit some details.

\begin{proof}[Proof of Proposition \ref{prop:newconductorK}]

If $q \neq r$ or if $q = r$ and $r \not \equiv 7 \pmod{8}$, the proof is virtually identical to that of Proposition \ref{prop:newconductorQ}, so let us assume that $q = r$ and $r \equiv 7 \pmod{8}$ with $v_r(s) \ge (r+3)/2$. By Lemmas \ref{lemma:auxramification} and \ref{prop:ramification}, the ramification index of the extension $L/K$ is now given by 
\begin{equation}
    e_{L/K} = \frac{e_{L/\Q_r}}{e_{K/\Q_r}} = \frac{(r-1)/2}{(r-1)/2} = 1.
\end{equation}
Consequently, the extension $K(\gamma_j)/K$ is unramified for any $j \neq k_0$ and there are precisely $r-1$ orbits of inertia in the set $\{\gamma_i \mid i \neq k_0\}$. Thus, the tame conductor is given by 
\begin{equation}
    \id{n}_{tame}(C/K) = r-1-r+1 = 0,
\end{equation}
and the curve therefore has good reduction over $K$. If $v_r(s) = (r+1)/2$, a similar argument by using Proposition \ref{prop:ramification} can be used to show that the ramification index of $K(\gamma_0)/K$ is $r$, so that 
\begin{equation}
    \id{n}_{tame}(C/K) = r-1-1+1 = r-1,
\end{equation}
which is the same result as in the case where $r \not \equiv 7 \pmod{8}$. All that remains is to add the wild conductor given by Proposition \ref{prop:wildconductor} to obtain that 
\begin{equation}
    \id{n}(C/K) = r-1 + \frac{r^2-1}{4} = \frac{r^2+4r-5}{4} = \frac{(r-1)(r+5)}{4}, 
\end{equation}
giving the desired result.
\end{proof}

\section{Irreducibility}\label{sec:irred}

In this section, we prove asymptotic irreducibility for the representation $\overline{\rho}_{J_{2,r}, \id{p}}$ defined in Section \ref{Sec:connectingfrey}. For this, we let $\id{r}$ be as in Section \ref{sec:notation}. Recall the following result (see \cite[Corollary 1]{BCDF1}).

\begin{thm}\label{thm:irredasym}
Let $K$ be a totally real field, $\id{q}$ a prime ideal of $K$ and $g$ a positive integer. Then there exists a constant $C(K,g,\id{q})$ such that for all primes $p>C(K,g,\id{q})$ and all $g$-dimensional abelian varieties $A/K$ satisfying:
\begin{enumerate}
    \item $A$ has potentially good reduction at $\id{q}$,
    \item $A$ is semistable at the primes of $K$ dividing $p$,
    \item $A$ is of $\GL_2(F)$-type for some totally real field $F$,
    \item all endomorphisms of $A$ are defined over $K$ (i.e.\
    $\End_K(A) = \End_{\overline{K}}(A)$),
\end{enumerate}
Then the residual representation $\bar{\rho}_{A,\id{p}}$ is irreducible.
\end{thm}

\begin{coro}\label{coro:large-image}
Assume that $(a,b,c)$ is a non-trivial primitive solution of $Ax^2+By^r=Cz^p$ satisfying:
\begin{enumerate}
    \item There exists a prime $\ell\neq 2, r$ such that $\ell\mid Cc$,
    \item $r\nmid Cc$.
\end{enumerate}
Then, there exists a constant $C(A,B,r)$, depending only on $A, B, r$, such that for every $p>C(A,B,r)$ and all $\id{p}\mid p$ in $K$, the residual representation of $\bar{\rho}_{J_{2,r},\id{p}}$ is absolutely irreducible.
\end{coro}

\begin{proof}
     We will show the statement by proving that the four conditions in Theorem \ref{thm:irredasym} hold. Since $r\nmid Cc$, $J_{2,r}$ has potentially good reduction at $\id{r}$ by Lemma \ref{lemma:pot-good-red-at-r}, thereby proving condition (1). Moreover,  from Theorem \ref{thm:2rp-GL2-type} we have that $J_{2,r}$  is of $\GL_2$-type over $K$ and so condition (3) is satisfied. Now, let $\ell \neq 2,r$ be a prime dividing $Cc$. By Proposition \ref{prop:mult-red-outside-2r}, it is of multiplicative reduction for $J_{2,r}$ and so condition (4) of the above theorem holds. In addition, condition (2) holds by Proposition \ref{prop:mult-red-outside-2r}, assuming that $p\nmid 2rAB$, which we can do simply by increasing the value of $C(A,B,r)$.

    The result then follows from Theorem \ref{thm:irredasym}, by taking $\id{q}=\id{r}$, and by noting that the genus of $J_{2,r}$ equals $(r-1)/2$ and $K=\Q(\zeta_r)^+$, both depending only on $r$.
\end{proof}

\begin{remark}
    Under general conditions, we are only able to prove asymptotic irreducibility, as per Corollary \ref{coro:large-image}. In many cases, such as the application that we shall study in Section \ref{sec:application}, we will be able to prove irreducibility for $p \ge 5$.
\end{remark}

\section{\texorpdfstring{On the equation $rx^2 +z^{2p}=y^r$}{On the equation rx2+z2p=yr}}\label{sec:application}

As an application of the theory developed in the previous sections we give some new results about the Conjecture \ref{con:application}. We recall that in \cite{LaradjiMignotteTzanakis11} the authors study the equation
\begin{equation}
    5x^2+q^{2n}=y^5,
\end{equation}
where $q$ is an odd prime number, $\gcd(x,y)=1$ and $n\geq 3$. We prove the following theorem which covers some of the remaining cases in Theorem \ref{thm:LaradjiMignotteTzanakis}.

\begin{theorem}\label{thm:diophantine_application}
Let $p\geq 7$ and $p\neq 11$. There is no integer solution $(a,b,c)$ with $\gcd(a,b,c)=1$ of the equation \begin{equation}\label{eq:diophantine_application}
    -5x^2 + y^5 = z^{2p},
\end{equation}
such that $2\mid a$ and $3\mid b$.
\end{theorem}

\begin{remark}
The condition $3\mid b$ is equivalent to $v\equiv \pm 1\pmod{3}$ in \cite[Theorem 1.1]{LaradjiMignotteTzanakis11}. Therefore, we have confirmed Conjecture \ref{con:application} when $v\equiv \pm 1\pmod{3}$. It seems that primes of the form $q=2000v^4 - 200v^2 + 1$ are not rare, for instance there exist $835$ primes with $1\leq v\leq 10^4$ and $3\nmid v$. For example, for $v=9997$ we get $q=19976010777852160201$, which is well beyond the bound $10^9$ given in \cite[Theorem 1.1]{LaradjiMignotteTzanakis11}.
\end{remark}

\begin{remark}\label{rem:5_mid_a}
It is not hard to show that $b$ is odd. With elementary number theory we can prove that there exist coprime $m,n\in\ZZ$, not both odd, such that
\begin{align*}
    a & = 5n(m^4-10m^2n^2 + 5n^4),\\
    c^p & = m(m^4 - 50m^2n^2 + 125n^4).
\end{align*}
Therefore, we get $5\mid a$.
\end{remark}

For the rest of the section we fix $r=5$. Therefore, we have $K=\QQ(\zeta_5)^+=\Q(\sqrt{5})$ and we denote by $\id{r}_5$ the unique prime of $K$ above $5$.

Equation (\ref{eq:diophantine_application}) is a particular case of (\ref{eqn:main}), with $r=5$, $A=-5$ and $B=C=1$. Then, from \eqref{eqn:frey},  to a solution $(a,b,c)$ to \eqref{eq:diophantine_application} we attach the Frey hyperelliptic curve $C:=C_{2,5}(a,b)$ given by
\begin{equation}
\label{eqn:Capplication}
    C:~y^2=x^5-25bx^3+125b^2x-250a.
\end{equation}
We denote by $J$  the Jacobian of $C$ and by $\{\rho_{J,\lambda}\}$ its associated compatible system of Galois representations. From Remark \ref{rem:5_mid_a} we have that $5\mid a$, hence from Theorem \ref{thm:conductor2rp} we know that the conductor exponent of $\{\rho_{J,\lambda}\}$ at $\id{r}_5$ is $2$.

Since Theorem \ref{thm:residualIrred} and Lemma \ref{lemma:irred-r=7} only cover $r \neq 5$ and \eqref{eq:diophantine_application} has signature $(2,5,p)$, we need the following extension to Theorem \ref{thm:residualIrred}.
\begin{proposition}\label{prop:brhoJf_absolutely_irreducible}
The representation $\brhoJf$ is absolutely irreducible when restricted to $G_{K(\zeta_5)}$.
\end{proposition}

\begin{proof}
From the proof of Theorem \ref{thm:residualtwistC2} we know that $\brhoJf$ satisfies
\begin{equation}\label{eq:isomorphism_to_C25}
    \brhoJf \simeq \overline{\rho}_{C_2, 5} \otimes \chi,
\end{equation}
as $G_K$-representations, where $\chi : G_K \to \overline{\F}_5^\times$ is a character of order at most two and $C_2:=C_2(1/t_0)$ for $t_0$ as defined in \eqref{eqn:t0s0}. Note that $C_2/\Q$ is an elliptic curve\footnote{It is non-singular as $ab(-5a^2+b^5)\neq 0$.} with the $2$-torsion point $(0,0)$, defined over $\Q$. As explained in \cite[\S III, Example 4.5]{Silverman}, $C_2(t)$ is $2$-isogenous to the curve 
\begin{equation}
    C_2'(t): y^2 = x^3-4x^2+4(1-t)x,
\end{equation}
whose $j$-invariant is equal to 
\begin{equation}
    j_{C_2'(t)}=64\frac{(3t+1)^3}{t(t-1)^2}.
\end{equation}
We will first show that the representation $\overline{\rho}_{J, \id{r}_5}$ is irreducible when restricted to $G_{K(\zeta_5)}$. Note that 
\begin{equation}
    j_{C_2'(t)}-1728 = \frac{64(9t-1)^2}{t(t-1)^2} = \frac{A_1^2}{t},
\end{equation}
where $A_1 = 8(9t-1)/(t-1)$. Now, let us set $t = 1/t_0 = c^{2p}/(-5a^2)$. Then,
\begin{equation}\label{eqn:invariantisogeny}
    j_{C_2'(1/t_0)}-1728 = -5A_2^2,
\end{equation}
with $A_2 = aA_1/c^p$. On the other hand, we know  that the $j$-map $j_5: X_0(5) \to X_0(1)$ is given by the expression 
\begin{gather}
    \label{eqn:jfunction}
\begin{aligned}
j_5 & = \frac{(u^2+4uv-v^2)^2(u^2+22uv+125v^2)}{v^5u} + 1728 \\ 
& =\frac{(h^2+4h-1)^2(h^2+22h+125)}{h} + 1728,
\end{aligned}
\end{gather}
where $h = u/v \in K(\zeta_5)$ (see \cite[Proof of Theorem 3.1]{Dahmen}). If $\overline{\rho}_{J, \id{r}_5}\mid_{G_{K(\zeta_5)}}$ were reducible, then from \eqref{eq:isomorphism_to_C25} it would correspond to a non-cuspidal $K(\zeta_5)$-rational point on the curve $X_0(5)$ and, consequently, \eqref{eqn:invariantisogeny} and \eqref{eqn:jfunction} yield that 
\begin{equation}
    -5A_2^2 = \frac{(h^2+4h-1)^2(h^2+22h+125)}{h}.
\end{equation}
Manipulating the equation, we find that 
\begin{equation}\label{eqn:twistedcurve}
    -5Y^2 = X(X^2+22X+125),
\end{equation}
where $Y = A_2h/(h^2+4h-1)$ and $X = h$.

Using \texttt{Magma}, we find that the above elliptic curve has only two $K(\zeta_5)$-rational points. Since $X_0(5)$ has precisely two rational cusps\footnote{See modular curve with label \href{https://beta.lmfdb.org/ModularCurve/Q/5.6.0.a.1/}{5.6.0.a.1} in LMFDB \cite{lmfdb}.} \cite{lmfdb}, we conclude that all the $K(\zeta_5)$-rational points of $X_0(5)$ are cuspidal, which is a contradiction.

Since we just showed that $\overline{\rho}_{J, \id{r}_5}\mid_{G_{K(\zeta_5)}}$ is irreducible, we proceed to prove the absolute irreducibility by following the argument in \cite[Proposition 6.5]{ChenKoutsianas}. Suppose that $\overline{\rho}_{J, \id{r}_5}\mid_{G_{K(\zeta_5)}}$ is not absolutely irreducible. Then, its image would lie in a non-split Cartan subgroup $H$. Hence, the image of $\overline{\rho}_{J, \id{r}_5}\mid_{G_{K}}$ lies in the normalizer of $H$ which implies that the representation $\brhoJf\mid_{G_{K}}$ corresponds to a non-cuspidal $K$-rational point of $X_{ns}^+(5)$. Therefore, we would have that
\begin{equation}\label{eqn:jinvariantequality}
    j_{C_2'(1/t_0)}-1728 = j_{5N}-1728,
\end{equation}
where $j_{5N}$ is the $j$-function from the curve $X_{ns}^+(5)$ to $X(1)$ \cite[Corollary 5.3]{Chen99}, given by
\begin{equation}\label{eq:j5N}
    j_{5N} = \frac{125s(2s+1)^3(2s^2+7s+8)^3}{(s^2+s-1)^5}.
\end{equation}
From \eqref{eqn:invariantisogeny} and \eqref{eqn:jinvariantequality} the pair $(s,y_0)$ is a $K$-rational point on the curve
\begin{equation}
    E_2:~-5y^2 = 2(x^2+x-1)\left(x^2+\frac{7}{2}x+\frac{27}{8}\right),
\end{equation}
where $y_0 = A_2(s^2+s-1)^3/(2^37(s^2+12s/7+8/7)(s^2+7s/2+1/4))$. 

The Mordell-Weil group of the Jacobian of $E_2$ over $K$ is isomorphic to $\Z/2\Z$. The non-trivial $2$-torsion point of the Jacobian correspond to the points $(s_0,0)$ and $(s_1,0)$ of $E_2$ where $s_0,s_1$ are the roots of the polynomial $\phi(s)=s^2+s-1$. From \eqref{eq:j5N} and the previous discussion, we understand that all the $K$-rational points of $E_2$ correspond to cuspidal points of $X_{ns}^+(5)$.
\end{proof}

\begin{lemma}\label{lem:cyclotomic_determinant}
The determinant of the representation $\rhoJp$ is the $p$-adic cyclotomic character.
\end{lemma}
\begin{proof}
This follows from \cite[Lemma 2.2.5 and Corollary 2.2.10]{WuThesis11}.
\end{proof}

\begin{proposition}\label{prop:modularity_application}
The abelian variety $J$ is modular.
\end{proposition}

\begin{proof}
By Proposition \ref{prop:brhoJf_absolutely_irreducible}, we know that the representation $\overline{\rho}_{J, \id{r}_5}|_{G_{K(\zeta_5)}}$ is absolutely irreducible. Then, the exact same proof as in Theorem \ref{thm:modularity} applies and successfully proves that the variety is modular.
\end{proof}

Although we are not able to compute the conductor exponent at $\id{q}_2$ in the general case in Section \ref{Sec:conductor}, we can completely determine it for our application following the ideas in \cite[Section 7]{ChenKoutsianas}.

\begin{lemma}\label{lem:potential_semistable_at_2}
The curve $C$ has potentially bad semistable reduction at $\id{q}_2$ and $\rhoJ\mid_{I_{K_{\id{q}_2}}}$  has special inertia type for $\lambda \nmid \id{q}_2$.
\end{lemma}

\begin{proof}
The Igusa invariants of $C$  are given by
\begin{align*}
    J_2 &=  17500b^2 = 2^2\cdot 5^4 \cdot 7 \cdot b^2,\\
    J_4 &= 8593750b^4, \\
    J_6 &= 25000000000a^2b - 117187500b^6, \\
    J_8 &= 109375000000000a^2b^3 - 18975830078125b^8, \\
    J_{10} &= 3125000000000000a^4 - 1250000000000000a^2b^5 + 125000000000000b^{10} = 2^{12}\cdot 5^{15}\cdot c^{2p}.
\end{align*}
According to \cite[Th\'{e}or\`{e}me 1 (I)]{Liu93} the curve $C$ has potentially good reduction at $\id{q}_2$ if and only if $J_{2i}^5/J_{10}^i$ is integral at $\id{q}_2$ for all $i=1,\cdots, 5$. We have that
\begin{equation}
    \frac{J_2^5}{J_{10}}=\frac{5^57^5b^{10}}{4c^{2p}}.
\end{equation}
Since $2\nmid b$, $J_2^5/J_{10}$ is not integral at $\id{q}_2$ and, consequently, $C$ has potential bad semistable reduction at $\id{q}_2$. Therefore, $\rhoJ\mid_{I_{K_{\id{q}_2}}}$ has special inertia type for $\lambda \nmid \id{q}_2$.
\end{proof}

Let $m$ be a rational prime and $S$ a finite set of (finite) primes of $L$, where $L$ is a number field. We denote by
\begin{equation}
    L(S,m)=\{d\in L^*/(L^{*})^{m}:~v_{\id{q}}(d)\equiv 0\pmod{m}~\text{for all}~\id{q}\in S\},
\end{equation}
the \textit{$m$-th Selmer group of $L$ with respect to $S$}. If $\zeta_m\in L$, let $L(S,m)^*$ be the set of characters $\chi:G_L\rightarrow \Z/m\Z$ corresponding to the extensions $L(\sqrt[m]{d})/L$, where $d\in L(S,m)$.

\begin{proposition}\label{prop:conductor}
Let $S_2=\{\id{q}_2\}$ and let $\lambda$ be a prime of $K$ not dividing $2$. Then there exists a suitable quadratic character $\chi_0 \in K(S_2, 2)^*$ such that the conductor exponent of $\rhoJ\otimes\chi_0$ at $\id{q}_2$ is equal to $5$ if $2\mid a$ or $6$ if $2 \mid c$.
In particular, $\id{n}(\rhoJ\otimes\chi_0)=\id{q}_2^s\id{r}_5^2\displaystyle\prod_{\id{q}\mid c}\id{q}$, where 
\begin{equation*}
    s = \begin{cases}
        5, & 2\mid a,\\
        6, & 2\mid c.
    \end{cases}
\end{equation*}
\end{proposition}

\begin{proof}
We consider the elliptic curve
\begin{equation}
    C_2'':~y^2 = x^3 - 10ax^2-5c^{2p}x,
\end{equation}
which is a quadratic twist of Darmon's curve $C_2$ in Theorem \ref{thm:residualtwistC2}. The conductor exponent of $C_2''$ at $\id{q}_2$, say $m$, is $5$ if $2\mid a$ and $6$ if $2\mid c$ due to Tate's algorithm\footnote{This is due to Step 4 in Tate's algorithm if $2\mid a$ and to Step 7 with type $I^*_{\nu(\nu\geq 4)}$, because $(v_2(c_4), v_2(c_6), v_2(\Delta))=(6,9,\geq 17)$ if $2\mid c$.} \cite{papado, TateAlgorithm}.

Because $C_2''$ is a quadratic twist of $C_2$, from the proof of Theorem \ref{thm:residualtwistC2} we understand that $\brhoJf$ satisfies
\begin{equation}
    \brhoJf \simeq \overline{\rho}_{C_2'', 5} \otimes \chi,
\end{equation}
where $\chi$ is a quadratic character of $G_K$. According to \cite[Lemma 7.14]{ChenKoutsianas} there exists a character $\chi_0\in K(S_2, 2)^*$ such that $\chi\chi_0$ is unramified at $\id{q}_2$. Therefore, the conductor exponent of $\brhoJf\otimes \chi_0$ at $\id{q}_2$ is $m$.

The representation $\brhoJf$ is absolutely irreducible by Proposition \ref{prop:brhoJf_absolutely_irreducible} and modular by Proposition  \ref{prop:modularity_application}. Because $\rhoJf$ has special inertia type, by Lemma \ref{lem:potential_semistable_at_2}, from \cite[Theorem 1.5]{Jarvis} we see that the representation $\rhoJf\otimes\chi_0$ does not degenerate at $\id{q}_2$, which implies that the conductor exponent of $\rhoJf\otimes\chi_0$ at $\id{q}_2$ is equal to that of $\brhoJf \otimes \chi_0$, which is $m$.

By modularity, we have that, for $\lambda \nmid 2$, the exponent of the compatible system $\{\rhoJ\otimes\chi_0\}$ at $\id{q}_2$ is equal to $5$ if $2\mid a$ or $6$ if $2\mid c$.
\end{proof}

\begin{proposition}\label{prop:Serre_level}
Let $p>5$ be a prime number. Then, there exists a suitable quadratic character $\chi_0\in K(S_2, 2)^*$ such that the Serre level of $\brhoJ\otimes\chi_0$ is equal to $\id{q}_2^s\id{r}_5^2$ where 
\begin{equation*}
    s = \begin{cases}
        5, & 2\mid a,\\
        6, & 2\mid c.
    \end{cases}
\end{equation*}
For all primes $\id{p}\mid p$ in $K$, the representation $\brhoJp\otimes\chi_0$ is finite at $\id{p}$.
\end{proposition}

\begin{proof}
Because $\chi_0$ is unramified outside $2$, for a prime $\id{q}$ not dividing $10$, the representation $\brhoJp\otimes\chi_0$ is unramified at $\id{q}\nmid p$ and finite at $\id{p}\mid p$ by \cite[Proposition 1.15]{DarmonDuke} and \cite[Section 7]{BillereyChenDieulefaitFreitas25a}.

From Proposition \ref{prop:conductor} we know that the conductor of $\rhoJp\otimes\chi_0$ is equal to $\id{q}_2^s\id{r}_5^2\prod_{\id{q}\mid c}\id{q}$. The Jacobian $J$ is modular by Proposition \ref{prop:modularity_application} and $\brhoJp\otimes\chi_0$ is irreducible by Proposition \ref{prop:irreducible_application}. Similarly to the proof of Proposition \ref{prop:conductor} we can prove that $\rhoJp\otimes\chi_0$ does not degenerate at $\id{q}_2$ and $\id{r}_5$. Then from level lowering \cite[Theorem 1.5]{Jarvis} we get that the conductor of $\brhoJp$ away from $\id{p}$ is equal to $\id{q}_2^s\id{r}_5^2$. Because the Serre level is independent of $\id{p}\nmid 10$ we get the result.
\end{proof}

For our practical applications, we need to show that the representation $\brhoJp$ is irreducible for any $p > 5$. For this, we follow the same strategy as in \cite[Theorem 8.6]{ChenKoutsianas}.

\begin{lemma}\label{lemma:potgoodreddegree4}
Let $K_{\id{q}_5}$ denote the completion of $K$ at the unique prime ideal $\id{q}_5$ dividing $5$, and let $M/K_{\id{q}_5}$ be any totally ramified extension of degree $4$. Then, $C/K_{\id{q}_5}$ attains good reduction over $M$. Moreover, the extension $M$ is minimal with respect to the ramification index.
\end{lemma}

\begin{proof}
Recall from \eqref{eqn:Capplication} that $C$ is given by 
\begin{equation}
    C:y^2 = x^5-25bx^3+125b^2x-250a,
\end{equation}
whose discriminant equals $\Delta_C=2^{12}5^{15}c^{4p}$. Let $\pi$ be a uniformizer of $M$ and $v$ its respective valuation. That is, $v(5)=8$  and so $v(\Delta_C)=120$. 
   
The change of variable $x\to x - 2\cdot 5^3a$ leads to the model
\begin{gather}
\begin{aligned}
y^2 & = x^5 - 2\cdot 5^4ax^4 + 5^2(25000a^2 - b)x^3 + 5^5(-50000a^3 + 6ab)x^2 \\ 
& + 5^3(156250000a^4 - 37500a^2b + b^2)x \\
& + 5^3(-7812500000a^5 + 3125000a^3b - 250ab^2 - 2a).
\end{aligned}
\end{gather}
Now, applying the change of variables $x\to \pi^6 x $, $y \to \pi^{15} y$ we get a new model $C'$, which is integral since $5 \mid a$, and whose discriminant is a unit in $M$ (see e.g.\ \cite[(3.9)]{ChenKoutsianas}). In particular, the special fibre of $C'$ is given by
\begin{equation}\label{eq:special_fibre}
    \tilde{C}':~y^2=x^5+\tilde{b}^2x,
\end{equation}
where $\tilde{b}$ is the image of $b$ in the residue field of $M$. Hence, $C'$ has good reduction over $M$.

It remains to prove the last statement. Suppose that $C$ attains good reduction over an extension $F/K_{\id{q}_5}$. Since $v_{\id{q}_5}(\Delta_C)=30$, then 4 must divide the ramification index of $F/K_{\id{q}_5}$, by \cite[(3.9)]{ChenKoutsianas}, and so $M$ is minimal with respect to the ramification degree.
\end{proof}

\begin{proposition}\label{prop:irreducible_application}
The representation $\overline{\rho}_{J, \id{p}}$ is absolutely irreducible for $p > 5$.
\end{proposition}

\begin{proof}
Since $\overline{\rho}_{J, \id{p}}$ is odd and the field $K$ is totally real, $\overline{\rho}_{J, \id{p}}$ is absolutely irreducible if and only if it is irreducible. To prove this, it suffices to see that the representation $\overline{\rho}:=\brhoJp \otimes \chi_0$ is irreducible, where $\chi_0$ is a quadratic character as in Proposition \ref{prop:Serre_level}.

From Lemma \ref{lem:cyclotomic_determinant} we have that $\det\brhoJp = \chi_p$ where $\chi_p$ is the $p$-adic cyclotomic character. Suppose $\overline{\rho}$ is reducible. Then
\begin{equation}
    \overline{\rho}\simeq 
    \begin{pmatrix}
        \theta & \star \\
        0 & \theta'
    \end{pmatrix},
\end{equation}
for some $\theta, \theta': G_K \to \mathbb{F}_\id{p}^*$ satisfying that $\theta\theta' = \chi_p$. Consequently, $\theta$ and $\theta'$ have the same conductor exponent away from $p$. We also have that the semi-simplification $\brho^{ss}$ of $\brho$ is isomorphic to $\theta\oplus\theta'$.

Let $\id{q} \neq \id{p}$ be a prime of $K$, and let $e_\id{q}$ be the conductor exponent of $\theta$ and $\theta'$ at $\id{q}$. We know that the conductor exponent of $\brho^{ss}$ at $\id{q}$ divides the conductor exponent of $\brho$ and the latter divides the conductor exponent of $\rhoJp\otimes\chi_0$. By Proposition \ref{prop:conductor}, it follows that $e_\id{q} = 0$ for all $\id{q} \neq \id{q}_2, \id{r}_5$, $e_{\id{q}_2} \leq 3$ and $e_{\id{r}_5}\leq 1$.

In addition, Lemma \ref{lemma:potgoodreddegree4} gives that $\rhoJp(I_{K_{\id{r}_5}})$ has order $4$. Because $\chi_0$ is unramified outside $2$ and the kernel of the reduction map  $\red:\imag\rhoJp\rightarrow\imag\brhoJp$ is a pro-$p$ group, and thus $\rhoJp(I_{K_{\id{r}_5}})$ does not intersect it, it follows that $\brho(I_{K_{\id{r}_5}})$ has order equal to $4$. As $\brho(I_{K_{\id{r}_5}})$ has order coprime to $p > 5$, Maschke's Theorem implies that $\overline{\rho}{|_{I_{K_{\id{r}_5}}}} \simeq (\theta \oplus \theta'){|_{I_{K_{\id{r}_5}}}}$. Therefore, the order of $\theta|_{{I_{K_{\id{r}_5}}}}$ and $\theta'|_{{I_{K_{\id{r}_5}}}}$ is $4$.

Since $\overline{\rho}$ is finite at $\id{p} \mid p$ and as $p$ is unramified in $K$, it follows from \cite[Corollaire 3.4.4]{Grothendieck} that $\theta \oplus\theta'$, when restricted to $I_{K_\id{p}}$, is isomorphic to either
\begin{equation}
    1 \oplus \chi_p \quad \text{or} \quad \psi_\id{p} \oplus \psi_\id{p}^p,
\end{equation}
where $\chi_p$ is the $p$-th cyclotomic character and $\psi_\id{p}$ is the fundamental character of level $2$ of the tame inertia group $I_{t, K_\id{p}}$. Next, we consider two cases.

Firstly, suppose that either $\theta$ or $\theta'$ are unramified at all primes $\id{p} \mid p$. Without loss of generality, suppose that $\theta$ is unramified. Then it corresponds to a character of the Ray class group of modulus $\id{q}_2^3\id{r}_5\infty_1\infty_2$, where $\infty_j$ denote the two places at infinity. The Ray class group of this modulus is isomorphic to $(\Z/2\Z)^3\times(\Z/4\Z)$. The characters $\theta$ and $\theta'$ cannot be quadratic characters because this would contradict the fact that $\#\theta({{I_{K_{\id{r}_5}}}})=\#\theta({{I_{K_{\id{r}_5}}}})=4$.

Therefore, we assume that $\theta$ has order $4$. We follow the proof of \cite[Theorem 2]{BCDF1} to bound $p$ in this case. Let $M=K(\phi)$, where $\phi^4 + (-2\sqrt{5} - 70)\phi^2 - 88\sqrt{5} + 260=0$. Then, $M$ is totally ramified at $5$. Let $\id{r}_5'\mid \id{r}_5$ be the unique prime ideal of $M$ above $\id{r}_5$. From Lemma \ref{lemma:potgoodreddegree4} the curve $C$ attains good reduction at $\id{r}_5'$ over $M$. Because $\Frob_{\id{r}_5'}=\Frob_{\id{r}_5}$ we get that
\begin{equation}
    P_{\id{r}_5}(X)\equiv (X-\theta(\Frob_{\id{r}_5}))(X-\theta'(\Frob_{\id{r}_5}))\pmod{\id{p}},
\end{equation}
where $P_{\id{r}_5}(X)$ is the characteristic polynomial of Frobenius $\Frob_{\id{r}_5}$. On the other hand, we have that $\theta^4(\Frob_{\id{r}_5})\equiv 1\pmod{\id{p}}$. From \eqref{eq:special_fibre} and computing the zeta function of $\tilde{C}'$ we can deduce that $P_{\id{r}_5}(X)=X^2 \pm 5$. From the above we have $p\mid \Res(X^4-1, X^2\pm 5)$ which implies that $p=2,3$.

Suppose now $\theta$ and $\theta'$ are both ramified at some prime above $p$. Then \cite[Corollary 8.5]{ChenKoutsianas} implies that $p$ divides the norm of $u-1$, where $u = \varepsilon_1^{12}$, $\varepsilon_1$ is the fundamental unit of $\cO_K$ and $N = 12$ is the smallest integer such that $\varepsilon_1^{N} \equiv 1 \pmod{\id{q}_2^3\id{r}_5}$ and $\varepsilon_1^{N}$ is totally positive. Since the norm of $u-1$ is $-320$, it follows that $p \le 5$, which is a contradiction.
\end{proof}

\begin{proof}[Proof of Theorem \ref{thm:diophantine_application}]

We recall the fact that
\begin{equation}
    \rhoJp\otimes\chi_0=\rhoJchip.
\end{equation}
From Propositions \ref{prop:modularity_application}, \ref{prop:irreducible_application} and Lemma \ref{lem:cyclotomic_determinant} the representation  $\brhoJchip$ is modular with trivial character and irreducible. By level lowering for Hilbert modular forms \cite{Jarv, Raj} there exists a Hilbert modular form of weight $2$, trivial character and level $\id{q}_2^5\id{r}_5^2$, from Proposition \ref{prop:Serre_level} since $2\mid a$, such that
\begin{equation}\label{eq:elimination_step_isomorphism}
    \brhoJchip\simeq \brhogp,
\end{equation}
where $\id{P}$ is a prime above $p$ in the eigenvalue field $K_g$ of $g$. 

For a prime $\id{q}\neq \id{q}_2,\id{r}_5,\id{p}$, we define
\begin{align}
    \label{eqn:congruencetraces}
    a_{\id{q}}(J\otimes\chi_0) & =\Tr\rhoJchip(\Frob_{\id{q}}), \\
    \label{eqn:congruencetraces2}
    a_{\id{q}}(g) & = \Tr\rhogp(\Frob_{\id{q}}),
\end{align}
where $\Frob_{\id{q}}$ is a Frobenius element at $\id{q}$. It holds 
\begin{align}
    a_{\id{q}}(g) & \equiv a_{\id{q}}(J\otimes\chi_0)\pmod{\id{p}}, & & \text{if}~\id{q}\nmid c,\\
    a_{\id{q}}(g)^2 & \equiv (N(\id{q}) + 1)^2\pmod{\id{p}}, & & \text{if }\id{q}\mid c,
\end{align}
where $N(\id{q})$ is the norm of $\id{q}$. We can avoid twisting by the character $\chi_0$ by observing that 
\begin{equation}\label{eq:trace_square}
    a_{\id{q}}(J\otimes\chi_0)^2=a_{\id{q}}(J)^2\chi_0(\Frob_{\id{q}})^2=a_{\id{q}}(J)^2.
\end{equation}

Because the isomorphism in \eqref{eq:elimination_step_isomorphism} is over $\overline{\F}_p$ we have to take into account the different embeddings of $K$ and $K_g$ into $\overline{\F}_p$. Therefore, if $L$ is the compositum of $K$ and $K_g$ we define
\begin{equation}
    B_{\id{q}}(g) = N(\id{q})\cdot N_{K/\Q}\left(a_{\id{q}}(g)^2  - (N(\id{q}) + 1)^2\right)\cdot\prod_{\substack{a,b\in\F_q,\\~-5a^2+b^5\neq 0}}N_{L/\Q}\left(\prod_{\sigma\in\Gal(K/\Q)}(a_{\id{q}}(g)^2 - (a_{\id{q}}(J)^\sigma)^2)\right).
\end{equation}
Then, by \eqref{eqn:congruencetraces}, \eqref{eqn:congruencetraces2} and \eqref{eq:trace_square}, $p\mid B_{\id{q}}(g)$. Taking the $\gcd$ of the $B_{\id{q}}(g)$ not equal to $0$ for different primes $\id{q}$, we get an upper bound of $p$. 

Using the primes $\id{q}$ above $3$, $7$, $11$, $13$, $17$, $19$ and $23$ we get that $p\leq 5$ for all but $4$ cases and $p\leq 23$ for the rest of newforms. For the remaining $4$ cases we use more primes, in particular $\id{q}\mid q$ for all $q\leq 100$ and $q\neq 2,5$, and we get $p=\{2,5, 11\}$, which finishes the proof. The total amount of computational time was approximately $21$ hours. 
\end{proof}

\bibliography{main.bib}
\bibliographystyle{plain}
\end{document}